\documentclass[12pt]{article}
\usepackage{amsmath,amsfonts,amsthm,amssymb,mathrsfs,mathtools,ulem}
\usepackage{graphicx}
\usepackage[usenames,dvipsnames]{color}
\usepackage{float}
\usepackage{graphicx}
\usepackage{subfigure}
\usepackage{caption}
\usepackage{cite}
\usepackage{mathrsfs}
\usepackage{url}
\usepackage{caption}
\usepackage{epstopdf}
\usepackage{hyperref}
\usepackage{comment}
\usepackage{color}

\newcommand{\R}{\mathbb{R}}

\newcommand{\dint}{\displaystyle\int}

\newcommand\redout{\bgroup\markoverwith
{\textcolor{red}{\rule[0.5ex]{2pt}{0.8pt}}}\ULon}

\theoremstyle{plain}
\newtheorem{theorem}{Theorem}[section]
\newtheorem{conjecture}{Conjecture}[section]
\newtheorem{corollary}[theorem]{Corollary}
\newtheorem{lemma}[theorem]{Lemma}
\newtheorem{proposition}[theorem]{Proposition}
\newtheorem{example}{Example}[section]

\theoremstyle{definition}
\newtheorem{definition}[theorem]{Definition}

\newtheorem{hyp}[theorem]{Hypothesis}

\newtheorem{remark}[theorem]{Remark}

\theoremstyle{remark}

\numberwithin{equation}{section}
\numberwithin{theorem}{section}

\usepackage{a4wide}
\usepackage{geometry}
\title{\vspace{-1cm}
Irregularity scales for Gaussian processes: Hausdorff dimensions and hitting probabilities}
\date{}
\begin{document}
\maketitle
\begin{center}
         \author{\vspace{-0.5cm} Youssef Hakiki\\
   \small{ Department of mathematics, Faculty of science Semlalia}, \\ Cadi Ayyad University,
   Marrakesh, Morocco, \\
   e-mail\textup{: \texttt{youssef.hakiki@ced.uca.ma}}}
 \end{center} 
 \vspace{0.01cm}
 \begin{center}
     \author{Frederi Viens\\
    \small Department of Statistics, Rice University,\\
    Houston, TX, USA,\\
     e-mail\textup{: \texttt{viens@rice.edu}}}
 \end{center}
\vspace{0.01cm}
\begin{abstract}
\small Let $X$ be a $d$-dimensional Gaussian process in $[0,1]$, where the component are independent copies of a scalar Gaussian process $X_0$ on $[0,1]$ with a given general variance function $\gamma^2(r)=\operatorname{Var}\left(X_0(r)\right)$ and a canonical metric $\delta(t,s):=(\mathbb{E}\left(X_0(t)-X_0(s)\right)^2)^{1/2}$ which is commensurate with $\gamma(t-s)$. Under a weak regularity condition on $\gamma$, referred to below as $\mathbf{(C_{0+})}$, which allows $\gamma$ to be far from H\"older-continuous, we prove that for any Borel set $E\subset [0,1]$, the Hausdorff dimension of the image $X(E)$ and of the graph $Gr_E(X)$ are constant almost surely. Furthermore, we show that these constants can be explicitly expressed in terms of $\dim_{\delta}(E)$ and $d$. However, when $\mathbf{(C_{0+})}$ is not satisfied, the classical methods may yield different upper and lower bounds for the underlying Hausdorff dimensions. This case is illustrated via a class of highly irregular processes known as logBm. Even in such cases, we employ a new method to establish that the Hausdorff dimensions of $X(E)$ and $Gr_E(X)$ are almost surely constant. The method uses the Karhunen-Lo\`eve expansion of $X$ to prove that these Hausdorff dimensions are measurable with respect to the expansion's tail sigma-field. Under similarly mild conditions on $\gamma$, we derive  upper and lower bounds on the probability that the process $X$ can reach the Borel set $F$ in $\mathbb{R}^d$ from the Borel set $E$ in $[0,1]$. These bounds are obtained by considering the Hausdorff measure and the Bessel-Riesz capacity of $E\times F$ in an appropriate metric $\rho_{\delta}$ on the product space, relative to appropriate orders. Moreover, we demonstrate that the dimension $d$ plays a critical role in determining whether $X\lvert_E$ hits $F$ or not. For this purpose, we introduce a further condition, denoted as \textbf{($C_{\ell}$)}, which is satisfied by all relevant examples from \textbf{($C_{0+}$)}. When $E$ is an Ahlfors-David-regular compact set in the metric $\delta$, we obtain precise upper and lower bounds on the hitting probability of $F$ by $X$ from $E$ in terms of Hausdorff measure and capacity in the Euclidean metric, utilizing specific kernels. These bounds facilitate the proof of an undecidability property, by which there are examples of sets $E\times F$ which have the same Hausdorff dimensions relative to $\rho_{\delta}$ but for which one target set $F$ has a positive hitting probability while the other does not.  
\end{abstract}

\vspace{0,2cm}


\section{Introduction}
This paper studies some fractal properties for Gaussian processes with a general covariance structure. Properties of interest include the Hausdorff dimension of the image sets and the graph sets, and corresponding hitting probabilities.
One of our motivations is to understand better the high path irregularity exhibited by certain Gaussian processes $X$ started from 0. For example the family of processes $X=B^{\gamma}$ defined in \cite{Viens&Mocioalca2005}, which for any given function $\gamma$ on $\R_+$ such that $\gamma^2$ is of class $\mathcal{C}^2$ on $\R_+$, with $\lim_0\gamma=0$, and $\gamma^2$ is increasing and concave near the origin, is defined by the following Volterra representation 
\begin{align}\label{volt repr}
    B^{\gamma}(t):=\dint_0^{t}\sqrt{\left(\frac{d\gamma^2}{dt}\right)(t-s)}dW(s),
\end{align}
where $W$ is a standard Brownian motion. 

In the particular case $\gamma(r):=\log^{-\beta}(1/r)$, where $\beta>1/2$, the process $B^{\gamma}$ is an element of the family of Gaussian processes called logarithmic Brownian motions (logBm). The condition $\beta>1/2$ ensures that $B^{\gamma}$ has continuous paths as guaranteed by the so-called Dudley-Fernique theorem (see for instance \cite{Adler 2}). This one-parameter family of logBm processes spans a wide range of highly irregular continuous Gaussian processes, which are not  H\"older-continuous. For general $\gamma$, the Dudley-Fernique theorem can be used  generically to show that $B^{\gamma}$ admits the function $h:r\mapsto \gamma(r)\log^{1/2}(1/r)$ as a uniform modulus of continuity almost surely, which is an indication of the non-H\"older-continuity of logBm. That property can in turn be established ``by hand''. Indications of how to do so are in Section 2, a full treatment being left to the interested reader. In any case, the logBm scale is instructive since it extends to the edge of continuous processes and beyond in a one-parameter family.


The broader model class defined via the Volterra representation \eqref{volt repr} is interesting and convenient for several reasons. It involves a simple kernel which makes it amenable to calculations. It produces a process $X=B^{\gamma}$  which, while not having stationary increments, has increments which are nonetheless roughly stationary. Proposition 1 in the original reference \cite{Viens&Mocioalca2005} explains how the canonical metric $\delta(s,t)$ of $X$, for $s,t\in \mathbb{R}_+$ is commensurate with $\gamma(t-s)$, for processes which are more irregular than the Wiener process, i.e. as soon as $r=o(\gamma^2(r))$. The variance of the process $X=B^{\gamma}$ at time $t$ is precisely $\gamma^2(t)$, which implies that the process starts at $0$, and that the scale of the process behaves similarly to the popular class of self-similar models, like fractional Brownian motion and related Gaussian processes, for which the variance equals $t^{2H}$ for self-similarity parameter $H$. Note for instance that the process $X=B^{\gamma}$ with $\gamma(r)=r^H$ yields a self-similar process known as the Riemann-Liouville fractional Brownian motion. It is $H$-self-similar, does not have stationary increments, but has increments whose variance is commensurate with the variance $|t-s|^{2H}$ of standard fractional Brownian motion (fBm). Aside from the fBm and logBm scales, many other scales of continuity can be obtained from $B^\gamma$, some of which yield interesting properties when examined from the lens of Hausdorff dimensions, as we will see. For instance, the choice $\gamma(x) = \exp ( - \log^{q} (1/x) )$, introduced at the end of Section 2.1, provides a process which is less irregular than logBm, but is more irregular than any H\"older-continuous process, such as fBm and Riemann-Liouville fBm. Again, this process does not have stationary increments, but it does satisfy the commensurability condition between $\delta$ and $\gamma$ (see Condition $\mathbf{(\Gamma)}$, i.e. the relations \eqref{commensurate} at the start of Section 2), and thus its increments can be deemed roughly stationary. Since this regularity scale defines processes which are intermediate between the extremely irregular logBm, and the H\"older-continuous processes, these  processes provide a good test of our methods' applicability. Interestingly, we will see that those processes share some desirable hitting probability features with H\"older-regular processes, which the logBm processes are too irregular to possess.

Most of the results in the literature about the fractal properties for Gaussian processes do not apply to the case of logBm, or to the processes which are more regular than logBm but not H\"older-continuous. For the question of  hitting probabilities, see for example \cite{Mueller-Tribe2002,Bierme&Xiao}; for the Hausdorff dimension of the image and the graph sets, see \cite{Hawkes}. This inapplicability stems from those references' assumptions which imply some form of H\"older continuity. To wit, the conditions in those references imply that, for some $\alpha\in (0,1)$, we have  $\gamma\left(r\right)\lesssim r^{\alpha}$ near the origin. To make matters more delicate yet, there are many regularity scales between the H\"older continuity scale and the logarithmic scale of logBm mentioned above, the aforementioned case of the choice $\gamma(x) = \exp ( - \log^{q} (1/x) )$ being only one such instance. 
This motivates us to study   
the fractal properties for Gaussian  processes $X$ with more general covariance structure, under flexible conditions which would encompass the entire class of a.s. continuous Volterra processes $B^\gamma$ in \eqref{volt repr}. 
We thus investigate these problems under some general conditions on the standard deviation function $\gamma$ only, with no direct reference to any regularity scale, and no assumption that our processes be given in a particular form such as the Volterra representation \eqref{volt repr}, so that our results may be satisfied by large classes of processes within and/or beyond the Hölder scale. We concentrate our efforts on handling the broadest possible class of processes which satisfy the commensurability condition $\delta(s,t) \asymp \gamma(|t-s|)$, namely Condition $\mathbf{(\Gamma)}$ from relations \eqref{commensurate}.  

By concentrating only on Condition $\mathbf{(\Gamma)}$, i.e. relations \eqref{commensurate}, we are able to relax the restriction of stationarity of increments (see Proposition 5 in \cite{Viens&Mocioalca2005}), and to break away from the confines of Hölder continuity, as illustrated above by the logBm class and other non-H\"older processes. Apart from the paper \cite{Eulalia&Viens2013}, and the original paper \cite{Viens&Mocioalca2005} where logBm was introduced, few authors have studied precise regularity results for Gaussian processes beyond the H\"older (fractional) scale. See \cite{NeufcourtViens} for a study of various regularity classes, some of which interpolate between logBm and fBm, in the context of central limit theorems for Gaussian time series with memory. Recently in \cite{GULISASHVILI}, logBm was proposed as a model for very rough volatility, making the ideas introduced in \cite{Rosenbaum et al} more quantitative when one leaves the H\"older scale. Recently, the logBm was employed to study the $\mathcal{C}^{\infty}$-Regularization of ODEs by noise as in \cite{Har-Perk}, the idea behind using logBm for this purpose is that the local time of logBm is highly regular (it is $\mathcal{C}^{\infty}$ in its space variable) due to the high irregularity of paths of the underlying process.
Another interesting class of Gaussian processes with non-stationary increments, which satisfy relations \eqref{commensurate}, are the evolution-sense solutions of the linear stochastic heat equation, see those studied in \cite{Ouahhabi&Tudor2013, TorresTudorViens}. The processes resulting from the models in those papers have complex H\"older regularity in space and in time, but stochastic heat equations driven by noises with logBm-type behavior or other non-H\"older noises, will have evolution-sense solutions which inherit those non-H\"older regularities. One has every reason to expect that these examples of processes will still satisfy Condition $\mathbf{(\Gamma)}$ (relations \eqref{commensurate}), which can be shown by employing arguments similar to the proof of Proposition 1 in \cite{Viens&Mocioalca2005}. 

These details are omitted, since the purpose of this paper is to remain at a scope which encompasses all these regularity scales simultaneously by requiring only the commensurability Condition $\mathbf{(\Gamma)}$, and interpreting our results via $\gamma$ only, not in reference to any specific scale. To be clear, the Volterra-type processes $B^\gamma$ in \eqref{volt repr} are convenient for generating examples of processes which satisfy Condition $\mathbf{(\Gamma)}$ and other general technical conditions. For instance, that logBm satisfies Condition $\mathbf{(\Gamma)}$ with $l=2$ was established in Proposition 1 in \cite{Viens&Mocioalca2005}. We will use such examples as illustrations, while our theorems and results are stated and established under more general conditions such as Condition $\mathbf{(\Gamma)}$. We now provide a summary of the results which we establish in this paper, and how they are articulated. 

In Section 2, we provide some general hypotheses on $\gamma$, which are important to ensure some desirable properties for the process $X$. 
Some preliminaries on Hausdorff measures, Bessel-Riesz capacities and Hausdorff dimension on $\R_+$ and $\R_+\times \R^d$, in a general context, are also given here. All these preliminaries allow us to provide optimal upper and lower bounds for the Hausdorff dimension of the image $X(E)$ and the graph $Gr_E(X)$, where $E\subset [0,1]$, and for the hitting probabilities estimates, in the sections 3 and 4 respectively. The choice to present results relative to subsets of $[0,1]$ in the time variable, as opposed to another time interval, is arbitrary, and used for convenience.  

Section 2 is also where we recall and establish important results on the process $X$ that imply lower bounds for hitting probabilities, and upper bounds for hitting probabilities and Hausdorff dimensions of images and of graphs. Those results are respectively Lemma \ref{lem two pts LND}, which proves a so-called 2-point local non-determinism property, and Lemma \ref{estim small ball 1}, which is a type of small-ball probability estimate (probability of reaching a small ball in space over a small ball in time of similar diameter). These are proved under the commensurability Condition $\mathbf{(\Gamma)}$, i.e. relations \eqref{commensurate}. Moreover we interpret these results under various general conditions on $\gamma$  which are not hard to check and are satisfied by large classes of regularity scales of interest to us and to others. 
With these tools in hands, and with the additional definitions and basic results recalled in Section 2 about Hausdorff dimensions relative to general metrics, we are able to provide the exact value of the Hausdorff dimension of the image $X(E)$ and the graph $Gr_E(X)$, where $E\subset [0,1]$, in the section 3, under mild regularity conditions which extend far beyond the H\"older scale. Similarly these tools help us provides some optimal lower and upper bounds for hitting probabilities in Section 4. The choice to present results relative to subsets of $[0,1]$ in the time variable, as opposed to another time interval, is arbitrary, and is used for convenience.

We finish this introduction with a detailed narrative description of the main results in Sections 3 and 4 and their ramifications.

Recall that in \cite{Hawkes}, Hawkes resolved the problem of computing the Hausdorff dimension of the image and of the graph of a Gaussian process $X$ with stationary increments, i.e. assuming $\delta(s,t)=\gamma(|t-s|)$, under the strong condition $\hbox{ind}_{*}(\gamma)>0$, where $\hbox{ind}_{*}(\cdot)$ is the lower index, which will be defined in \eqref{def index}. A positive lower index for $\gamma$ implies $\alpha$-H\"older-continuity of the paths of $X$ for all $\alpha \in (0,\hbox{ind}_{*}(\gamma))$. In section 3, we relax those two conditions used by Hawkes. We consider functions $\gamma$ which satisfy a very mild regularity condition: the general condition labeled as $(\mathbf{C_{0+}})$, by which the inequality \eqref{nice condition}) holds for all $\varepsilon \in (0,1)$. Assuming these, using methods from potential theory and covering arguments, we prove in Section 3.1 that for all Borel set $E\subset [0,1]$, the Hausdorff dimension of $X(E)$ and $Gr_E(X)$ are constants almost surely, which are provided explicitly in terms of $\dim_{\delta}(E)$ and $d$, where $\dim_{\delta}(\cdot)$ denotes the Hausdorff dimension associated with the canonical metric $\delta$, and $d$ is the dimension of the ambient image space.  In this same Section 3.1 we also show in Lemma \ref{lem check nice cond} that the condition ``$\hbox{ind}_{*}(\gamma)>0$'' used by Hawkes implies the regularity condition $(\mathbf{C_{0+}})$; however, we also know from Example \ref{exmpl weak cond} that condition $(\mathbf{C_{0+}})$ goes significantly further than ``$\hbox{ind}_{*}(\gamma)>0$'' since  it is satisfied by the aforementioned important regularity class where $\gamma(x) = \exp ( - \log^{q} (1/x) )$, for which $\hbox{ind}_{*}(\gamma)=0$.

On the other hand, in some regularity scales, condition $(\mathbf{C_{0+}})$ fails to hold. Without this condition, the method of using potential theory and covering arguments may lead to different upper and lower bounds for the Hausdorff dimension, both for the image and for the graph of $X$. For instance, in the logBm case, $(\mathbf{C_{0+}})$ fails because \eqref{nice condition} holds only for some, though not all, $\varepsilon \in (0,1)$. 
Therefore, in Section 3.2, we develop a general method that enables us to prove that the Hausdorff dimension of the image and of the graph are almost surely constants, which hold for any continuous Gaussian process $X$. The idea we introduce is to use the Karhunen-Lo\`eve representation of $X$ and to prove that, for any Borel set $E\subset [0,1]$ the Hausdorff dimensions of $X(E)$ and $Gr_E(X)$ are measurable with respect to certain tail sigma-fields, so we can apply  a Kolmogorov zero-one law, showing that these random variables are almost surely constants. These constants depends on $E$, and when $(\mathbf{C_{0+}})$ fails, they  are not given explicitly, but for example, in the scale of logBm, the upper and lower bounds which we obtain with the capacity+chaining method are explicit and become nearly optimal towards the upper end of the logBm scale, i.e. when $\beta\gg 1/2$. To be specific, for instance in the case of the graph's dimension, while Section 3.1 shows using a general argument that Condition $(\mathbf{C_\varepsilon})$ for fixed $\varepsilon$ implies, for an appropriate metric $\rho_{\delta}$ defined in \eqref{parabolic metric}, that $\dim_{\rho_{\delta}}\left(Gr_E(X)\right)$ is bounded below by $\dim_{\delta}(E)$ and above by $ \dim_{\delta}(E)+\varepsilon\,d$, in Section 3.2, in the specific case where $\gamma(r)$ is commensurate with $\log ^{-\beta}(1/r)$, a slightly finer analysis implies that the upper bound can be replaced by $\dim_{\delta}(E)\beta/(\beta-1/2)$. When $\beta$ is large, i.e. towards the higher regularity range of logBm, this is equivalent to $\dim_{\delta}(E)(1+1/(2\beta))$. The factor $\beta/(\beta-1/2)$ is not an improvement over the general result in Section 3.1 on the lower end of the logBm scale, since it explodes when $\beta$ approaches $1/2$, but it is an improvement on the general result when the logarithmic-scale Hausdorff dimension $\dim_{\log}(E)$ is finite (see equation \eqref{dimlog} and following line for the definition and relevant property of $\dim_{\log}(\cdot)$). Indeed, the $\gamma$ of the logBm scale satisfies Condition $(\mathbf{C_{\varepsilon}})$ with $\varepsilon = 1/(2\beta)$ and, noting that $ \dim_{\delta}(E) = \beta^{-1} \dim_{\log}(E)$, where $\dim_{\log}(E)$ is intrinsic to $E$ (i.e. does not depend on $\beta$), thus one only needs to require $\dim_{\log}(E)<\beta\,d$ to get an improved upper bound. This requirement, and the corresponding improvement on the upper bound, which incidentally is dimension-independent, holds for large $\beta$ as soon as $\dim_{\log}(E)$ is finite.

In section 4, our investigation focuses on the hitting-probabilities problem, i.e. estimating the probability of the event $\left\{X(E)\cap F\neq \varnothing \right\}$ where $E\subset [0,1]$ and $F\subset \mathbb{R}^d$ are Borel sets. Assuming that functions $\gamma$ satisfy Condition $(\mathbf{C_\varepsilon})$ for some fixed $\varepsilon \in (0,1)$ and a slightly strengthened concavity condition near the origin (Hypothesis \ref{Hyp2}), again using the capacity+chaining method, we obtain upper and lower bounds on the probability in question in terms of the Hausdorff measures and the Bessel-Riesz capacities of $E \times F$, relative to appropriate metrics and orders. These results are estalished in the first subsection of Section 4. These bounds suggest that, under condition $(\mathbf{C_{0+}})$, the dimension $d$ is a critical value for the dimension of $E \times F$ in the intrinsic metric. 

In the second subsection of Section 4, we do in fact prove that under the slightly stronger condition $(\mathbf{C_{0}})$, we can improve our results quantitatively, by making mild regularity assumptions (Ahlfors-David regularity) on either the set $E$ or the set $F$. We show in this subsection that the aforementioned criticality follows, by proving that, for any process $X$ satisfying a condition $(\mathbf{C_{\ell}})$, defined therein, which is an intermediate condition between the weaker condition $(\mathbf{C_{0+}})$ and the stronger condition $(\mathbf{C_{0}})$, whether or not a set can be reached by $X$ with positive probability cannot be decided when the dimension of $E \times F$ is critical. This condition is satisfied by all our examples of functions $\gamma$ of interest with zero index satisfying $(\mathbf{C}_{0+})$. In particular, the case $\gamma(x)=\exp\left(-\log^{q} (1/x) \right)$ with $q\in (0,1)$ satisfies $(\mathbf{C_{\ell}})$. We provide references in Section 4 and we explain therein how our results improve on prior known criticality studies, where processes $X$ were restricted to being H\"older-continuous and sets $E$ were restricted to being intervals. 

As a final application of our general result on hitting probabilities, in the last subsection of Section 4, we show first, under condition $(\mathbf{C_{0+}})$, that the so-called stochastic co-dimension of $X(E)$ exists and is given by $d-\dim_{\delta}(E)$ under a mild regularity condition on $E$. On the other hand, when condition $(\mathbf{C_{0+}})$ fails to hold, the method may lead to some upper and lower bounds for the hitting probabilities which are not necessarily optimal. We use the logBm case to illustrate this lack of optimality. In this case, the hitting-probabilities estimates do not help to compute the stochastic co-dimension of $X(E)$. However, since we proved in Section 3 that the Hausdorff dimension of $X(E)$ is almost surely constant, denoting this constant by $\zeta(E)$, then it is well within the realm of the possible, under some regularity condition on $E$ (e.g. similar to the Ahlfors-David regularity), that the stochastic co-dimension of $X(E)$ might be equal to $d-\zeta(E)$. This is an open problem at this point, and we do not have a well-developed strategy to resolve it, leaving it as a conjecture.      

\section{Preliminaries}\label{PrelimSec}
This section collects and establishes general facts about Gaussian processes whose variance function $\gamma^2$ is an increasing function starting from $0$,  particularly those whose canonical metric is commensurate with $\gamma$, a property referred to below as Condition $\mathbf{(\Gamma)}$ given by relations \eqref{commensurate}. The key technical estimate for upper bounds on Hausdorff measures of images and graphs is Lemma \ref{lem upper bound} below. It holds without any regularity assumptions on $\gamma$. We provide mild technical conditions which imply various levels of regularity, including corresponding estimates of the integral $f_{\gamma}$ featured in this lemma. Examples illustrating the various regularity behaviors are provided. Lemma \ref{lem two pts LND} is a two-point local non-determinism property which will help us establish lower bounds on hitting probabilities. It assumes a mild concavity property near the origin, referred to below as Hypothesis \ref{Hyp2}.

The second part of this section provides the definitions of Hausdorff measures and Riesz-Bessel capacities needed to understand and quantify the results in this paper. Since we work beyond H\"older regularity scales, notions of capacities and Hausdorff measures with respect to power functions apply when modified to be relative to non-H\"older metrics, using balls and distances relative to our processes' regularity scales, e.g. the processes' canonical metrics rather than powers of Euclidean distance; Hausdorff dimensions are thus relative to those metrics. General results expressing equivalent formulations of these Hausdorff dimensions are collected and justified in this section. Some of our results later in the paper will also relate to Euclidean-metric Hausdorff dimensions. 

\subsection{Gaussian processes with general variance function and commensurate squared canonical metric}
 In this entire paper we will work with $ \{X_0(t),t\in \mathbb{R}_+\}$ a real-valued mean-zero continuous Gaussian process defined on a complete  probability space $(\Omega,\mathcal{F},\mathbb{P})$, with canonical metric $\delta$ of $X_0$ on $(\mathbb{R}_+)^2$ defined by 
 $$ \delta(s,t):=\left(\mathbb{E}(X_0(s)-X_0(t))^2\right)^{1/2}.
 $$ 
 Let $\gamma$ be continuous increasing function on $\mathbb{R}_+$ (or possibly only on a neighborhood of $0$ in $\mathbb{R}_+$), such that $ \lim_{0+} \gamma =0$. We assume the following throughout, which we refer to as Condition $\mathbf{(\Gamma)}$: for some constant $l\geq 1$ we have, for all $s,t \in \mathbb{R}_+$, or possibly only all $s,t$ in the neighborhood of $0$ where $\gamma$ is defined,
\begin{align}
\mathbf{(\Gamma) :} \left\{\begin{array}{rcl}
    &\mathbb{E}\left(X_0(t)\right)^2=\gamma^2(t) \\
    \\
   & \text{and} \\
   \\
     &1/\sqrt{l}\,\gamma\left(|t-s|\right)\leq \delta(t,s)\leq \sqrt{l}\,\gamma(|t-s|).    
\end{array}\right.\label{commensurate}
\end{align}
Now, we consider the $\mathbb{R}^d$-valued process $X=\{X(t): t\in \mathbb{R}_+\}$ defined by 
\begin{align}\label{d iid copies}
    X(t)=(X_1(t),...,X_d(t)),\quad t\in \mathbb{R}_+,
\end{align}
where $X_1,...,X_d$ are independent copies of $X_0$. Let us consider the following hypotheses
 
 \begin{hyp}\label{H1} The increasing function $\gamma$ is concave in a neighborhood of the origin, and for all $0<a<\infty$, there exists $\varepsilon>0$ such that $\gamma^{\prime}(\varepsilon +)>\sqrt{l}\, \gamma^{\prime}(a{-})$.
 \end{hyp}
 
 \begin{hyp} \label{Hyp2} For all $0<a<b<\infty$, there exists $\varepsilon>0$ and $\mathsf{c}_0\in (0,1/\sqrt{l})$, such that
\begin{align}\label{Condition Hyp2}
     \gamma(t)-\gamma(s)\leq \mathsf{c}_0\gamma(t-s) \quad \text{ for all $\,s,t\in [a,b]\,$ with $\,0<t-s \leq \varepsilon$}.
\end{align}
 \end{hyp}
 The following lemma shows that Hypothesis \ref{H1} implies Hypothesis \ref{Hyp2}, and under the strong but typical condition $\gamma^{\prime}(0+)=\infty$, the constant $\mathsf{c}_0$ in \eqref{Condition Hyp2} can be chosen arbitrarily small. The proof is given in \cite{Eulalia&Viens2013}.
 \begin{lemma}\label{lem Hyp1 implies Hyp2}
     Hypothesis \ref{H1} implies Hypothesis \ref{Hyp2}. Moreover if $\gamma^{\prime}(0+)=+\infty$, then for all $0<a<b<\infty$ and all $\mathsf{c}_0>0$, there exists $\varepsilon>0$ such that 
     $$
     \gamma(t)-\gamma(s)\leq \mathsf{c}_0\, \gamma(t-s) \quad \text{ for all $\,t,s \in [a,b]\,$ with $\,0<t-s<\varepsilon$}.
     $$
 \end{lemma}
 The following lemma is also proven in \cite{Eulalia&Viens2013}. 
\begin{lemma}\label{lem two pts LND}
Assume Hypothesis \ref{Hyp2}. Then for all $0<a<b<\infty$, there exist constants $\varepsilon>0$ and $\mathsf{c}_1>0$ depending only on $a,b$, such that for all $s,t\in [a,b]$ with $|t-s|\leq \varepsilon$,
\begin{align}\label{two pts LND}
    Var\left(X_0(t)\lvert X_0(s)\right)\geq \mathsf{c}_1\,\delta^2(s,t) \geq (\mathsf{c}_1/l)\,\gamma^2(|t-s|).
\end{align}
\end{lemma}
\noindent Condition \eqref{two pts LND} is called {\it two-point local non-determinism}.

We denote by $B_{\delta}(t,r)=\{s\in \mathbb{R}_+: \delta(s,t)\leq r\}$ the closed ball of center $t$ and radius $r$ in the metric $\delta$. 
The following lemma is useful for the proof of the upper bounds for the Hausdorff dimension in Theorem \ref{Hausd dim image}. It is an improvement of both of proposition 3.1. and proposition 4.1. in \cite{Eulalia&Viens2013}. The proof that we give here uses similar arguments to those of \cite[Proposition 4.4.]{Dal-Khosh-Eulal 2007}.
\begin{lemma}\label{lem upper bound}
Assume that $\gamma$ satisfies the commensurability condition $\mathbf{(\Gamma)}$, i.e. relations \eqref{commensurate}. Let $0<a<b<\infty$, and $I:=[a,b]$.
Then for all $M>0$, there exist positive constants $\mathsf{c}_2$ and $r_0$ such that for all $r\in (0,r_0)$, $t\in I$  and $z\in [-M,M]^d$ we have 
\begin{equation}\label{estim small ball 1}
\mathbb{P}\left\{\inf _{ s\in B_{\delta}(t,r)\cap I}\|X(s)-z\| \leqslant r\right\} \leqslant \mathsf{c}_2 (r+f_{\gamma}(r))^{d},
\end{equation}
		where $\|\cdot\|$ is the Euclidean metric, and $f_{\gamma}$ is defined by $$f_{\gamma}(r):= \dint_0^{1/2}\frac{\gamma\left(\gamma^{-1}({l}^{1/2}\,r)\,y\right)}{y \sqrt{\log(1/y)}}dy.$$ 
\end{lemma}
\begin{proof}
We begin by observing that, for all $M>0$ and $z=(z_1,\ldots,z_d)\in [-M,M]^d$, we have 
$$
\left\{\inf_{ s\in B_{\delta}(t,r)\cap I}\|X(s)-z\| \leqslant r\right\}\subseteq \bigcap_{i=1}^{d} \left\{\inf_{ s\in B_{\delta}(t,r)\cap I}\lvert X_i(s)-z_i\rvert \leqslant r\right\}.  
$$
Then since the coordinate processes of $X$ are independent copies of $X_0$, it is sufficient to prove  \eqref{estim small ball 1} for $d=1$. Note that for any $s,t\in I$, we have
\begin{align}
\mathbb{E}\left(X_{0}(s) \mid X_{0}(t)\right)=\frac{\mathbb{E}\left(X_{0}(s) X_{0}(t)\right)}{\mathbb{E}\left(X_{0}(t)^{2}\right)} X_{0}(t):=c(s, t) X_{0}(t).
\end{align}
This implies that the Gaussian process $(R(s))_{s\in I}$ defined by 
\begin{align}
    R(s):=X_{0}(s)-c(s, t) X_{0}(t),\label{process minus projection}
\end{align}
is uncorrelated with and thus independent of $X_{0}(t)$, since these two processes are jointly Gaussian. Let
$$
Z(t, r):=\sup _{s \in B_{\delta}(t, r) \cap I}\left|X_{0}(s)-c(s, t) X_{0}(t)\right|.
$$
Then
\begin{align}
\begin{aligned}
\mathbb{P}&\left\{\inf _{s \in B_{\delta}(t, r) \cap I}\left|X_{0}(s)-z_0\right| \leq r\right\} \\
&\leq \mathbb{P}\left\{\inf _{s \in B_{\rho}(t, r) \cap I}\left|c(s, t)\left(X_{0}(t)-z_0\right)\right| \leq r+Z(t, r)+\sup _{s \in B_{\delta}(t, r) \cap I}|(1-c(s, t)) z_0|\right\}
\end{aligned}\label{upper bound 1}
\end{align}
By the Cauchy-Schwarz inequality and relations \eqref{commensurate}, we have for all $s, t \in I$,
\begin{align}\label{estim correlation}
    \begin{aligned}
        |1-c(s, t)|&=\frac{\left|\mathbb{E}\left[X_{0}(t)\left(X_{0}(t)-X_{0}(s)\right)\right]\right|}{\mathbb{E}\left(X_{0}(t)^{2}\right)} \\
        &\leq \frac{\left(\mathbb{E}( X_0(t))^2\right)^{1/2}\left(\mathbb{E}( X_0(t)-X_0(s))^2\right)^{1/2}}{\mathbb{E}\left(X_{0}(t)^{2}\right)}= \frac{\delta(s,t)}{\gamma(t)}\\
        &\leq\, \mathsf{c}_3\,\delta(s,t),
    \end{aligned}
\end{align}
where $\mathsf{c}_{3}=(\gamma(a))^{-1}$. Let $r_0:=1/2\mathsf{c}_3$, then \eqref{estim correlation} implies that for all $0<r<r_0$ and $s\in B_{\delta}(t,r)\cap I$, we have $1/2\leq c(s,t)\leq 3/2$. Furthermore, for $0<r\leq r_0$, $s\in B_{\delta}(t,r)$, and $z_0\in [-M,M]$, we have 
$$
|(1-c(s,t))z_0|\leq\,\mathsf{c}_3\,M\,r.
$$
Combining this inequality with \eqref{upper bound 1}, we derive that
\begin{align}\label{bnd small ball via regression}
\begin{aligned}
    \mathbb{P}\left\{\inf _{s \in B_{\delta}(t, r) \cap I}\left|X_{0}(s)-z\right| \leq r\right\}&\leq \mathbb{P}\left\{\left|X_0(t)-z\right|\leq 2\left(\mathsf{c}_3\,M+1\right)r+2 Z(t,r)\right\}\\
    &\leq \mathsf{c}_4\left(r+\mathbb{E}\left[Z(t,r)\right]\right),
\end{aligned}
\end{align}
for all $z_0\in [-M,M]$ and $0<r<r_0$, where the constant $\mathsf{c}_4$ depends on $M$, $a$, $b$, $l$ and $\mathsf{c}_3$ only. The last inequality follow from the independence between $X_0(t)$ and $Z(t,r)$. 

Now we bound $\mathbb{E}\left[Z(t,r)\right]$. Indeed, we have 
\begin{align}
    Z(t,r)\leq  Z_{1}(t,r)+Z_{2}(t,r),
\end{align}
where 
\begin{align*}
    \begin{aligned}
        Z_{1}(t,r)&:=\lvert X_0(t)\rvert \sup_{s\in B_{\delta}(t,r)\cap I}\lvert 1-c(s,t)\rvert\\
        Z_{2}(t,r)&:=\sup_{s\in B_{\delta}(t,r)\cap I}\lvert X_0(s)-X_0(t)\rvert.
    \end{aligned}
\end{align*}
Using \eqref{estim correlation} and Cauchy-Schwartz inequality we get that 
\begin{align}\label{bnd Z1}
\mathbb{E}\left[Z_{1}(t,r)\right]\leq \mathsf{c}_5\,r,
\end{align}
where $\mathsf{c}_5:=\gamma(b)/\gamma(a)$. Recall that relations \eqref{commensurate} ensure that $B_{\delta}(t,r)\subseteq \{s\in \R_+\,:\, |t-s|\leq \gamma^{-1}({l}^{1/2}\,r) \}$. Therefore
$$
Z_{2}(t,r)\leq \sup_{\substack{|t-s|\leq \gamma^{-1}({l}^{1/2}\,r)\\ s\in I}}\lvert X_0(t)-X_0(s)\rvert .
$$
Now, using the fact that $ \delta(s,t)\leq \sqrt{l} \gamma(|t-s|)$ then \cite[Lemma 7.2.2]{Marcus-Rosen} ensures that 
\begin{align}\label{bnd Z2}
\begin{aligned}
\mathbb{E}\left[Z_2(t,r)\right]&\leq\, \mathbb{E}\left[\sup_{\substack{|t-s|\leq \gamma^{-1}({l}^{1/2}\,r)\\ s\in I}}\lvert X_0(t)-X_0(s)\rvert\right]\\
&\leq\, \mathsf{c}_6\, \left(\gamma\left(\gamma^{-1}({l}^{1/2}\,r)\right)+\dint_0^{1/2} \frac{\gamma\left(\gamma^{-1}({l}^{1/2}r)\,y\right)}{y\, \log^{1/2}(1/y)}dy\right)\\
&\leq \, \mathsf{c}_7\, \left(r+f_{\gamma}(r)\right),
\end{aligned}
\end{align}
where $\mathsf{c}_{6}$ is a universal constant which depends on $l$ only, and $\mathsf{c}_7=\sqrt{l}\,\mathsf{c}_6$. Combining \eqref{bnd small ball via regression},...,\eqref{bnd Z2} the desired upper bound \eqref{estim small ball 1} follows immediately.
\end{proof}
\noindent Lemma \ref{lem upper bound} is quantitatively efficient when $r$ and $f_\gamma(r)$ are of the same order as $r\to 0$. The following condition \textbf{$(\mathbf{C_0})$} describes this situation:

\textbf{$(\mathbf{C_0})$}: 
There exist two constants $\mathsf{c}_8>0$ and $x_0\in (0,1)$ such that 
\begin{align}\label{condition raisonable}
\int_{0}^{1 / 2} \gamma(x y) \frac{d y}{y \sqrt{\log (1 / y)}} \leq \mathsf{c}_8\, \gamma(x)\quad \text{ for all $x\in [0,x_0]$}.
\end{align}
\begin{corollary}\label{cor estim small ball}
If $\gamma$ satisfies the condition \textbf{$(\mathbf{C_0})$}, then for all $M>0$, there exists some constant $\mathsf{c}_9$ depending on $\gamma$, $I$, $r_0$, $x_0$ and $M$, such that for all $z\in [-M,M]^d$ and for all $r\in (0,r_0 \wedge \gamma(x_0))$ we have 
\begin{align}
\mathbb{P}\left\{\inf _{ s\in B_{\delta}(t,r)\cap I}\|X(s)-z\| \leqslant r\right\} \leqslant \mathsf{c}_9\, r^{d}.\label{estim small ball 2}
\end{align}
\end{corollary}
 It is immediate that all power functions satisfy \eqref{condition raisonable}. Moreover, we will see in the sequel that \eqref{condition raisonable} is satisfied by all regularly varying functions of index $\alpha\in (0,1]$. We include some facts here about indexes for the reader's reference.
    
Let $\gamma:(0,1]\rightarrow \R_+$ be a continuous function which is increasing near zero and $\lim_{x\downarrow 0}\gamma(x)=0$. Then its lower and upper indexes $\mathrm{ind}_{*}(\gamma)$ and $\mathrm{ind}^{*}(\gamma)$ are defined respectively as 
\begin{align}\label{def index}
    \begin{aligned}
    \mathrm{ind}_{*}\left(\gamma\right):&=\sup\{\alpha: \gamma(x)=o\left(x^{\alpha}\right)\}\\
    &=\left(\inf\{\beta: \gamma(x)=o\left(x^{1/\beta}\right)\} \right)^{-1}.
    \end{aligned}
\end{align}
and
\begin{align}
    \begin{aligned}
 \quad\quad \mathrm{ind}^{*}\left(\gamma\right)&:= \inf\left\{ \alpha \geq 0: \, x^{\alpha}=o\left(\gamma(x)\right)  \right\}\\
    &=\sup\left\{\alpha\geq 0:\, \liminf_{x\downarrow 0}\left(\frac{\gamma(x)}{x^{\alpha}}\right)=0 \right\}.
    \end{aligned}
\end{align}
It is well known that $\mathrm{ind}_{*}(\gamma)\leq \mathrm{ind}^{*}(\gamma)$. Moreover we have the following statement
\begin{lemma}\label{bnd for lwr-uppr ind}
    If $\gamma$ is differentiable near $0$, then
\begin{align}\label{bnd on indexes} \mathrm{ind}_{*}\left(\gamma\right)\geq \liminf_{r\downarrow 0}\left(\frac{r\, \gamma^{\prime}(r)}{\gamma(r)}\right)  \quad \text{ and } \quad \mathrm{ind}^{*}\left(\gamma\right)\leq \limsup_{r\downarrow 0}\left(\frac{r\, \gamma^{\prime}(r)}{\gamma(r)}\right).
    \end{align}
    \end{lemma}
    \begin{proof}
We start with the left hand term of \eqref{bnd on indexes}. We assume that $\liminf_{r\downarrow 0}\left(r\, \gamma^{\prime}(r)/\gamma(r)\right)>0$ otherwise there is nothing to prove. Let us fix $0<\alpha^{\prime}<\alpha<\liminf_{r\downarrow 0}\left(r\, \gamma^{\prime}(r)/\gamma(r)\right)$, then there is $r_0>0$ such that $\alpha/r\leq \gamma^{\prime}(r)/\gamma(r)$ for any $r\in (0,r_0]$. Next, for $r_1<r_2\in (0,r_0]$ we integrate over $[r_1,r_2]$ both of elements of the last inequality, we obtain that $\log\left(r_2/r_1\right)^{\alpha}\leq \log\left(\gamma(r_2)/\gamma(r_1)\right)$, this implies immediately that $r\mapsto \gamma(r)/r^{\alpha}$ is nondecreasing on $(0,r_0]$, and thence $\lim_{r\downarrow 0}\gamma(r)/r^{\alpha}$ exists and finite. Since $\alpha^{\prime}<\alpha$, we get $\lim_{r\downarrow 0}\gamma(r)/r^{\alpha^{\prime}}=0$ and then $\alpha^{\prime}\leq \mathrm{ind}(\gamma)$. Since $\alpha^{\prime}$ and $\alpha$ are arbitrarily chosen, the desired inequality holds by letting $\alpha^{\prime}\uparrow \alpha$ and $\alpha\uparrow \liminf_{r\downarrow 0}\left(r\, \gamma^{\prime}(r)/\gamma(r)\right)$.     
    
For the upper inequality in \eqref{bnd on indexes}, we assume that $\limsup_{r\downarrow 0}\left(r\, \gamma^{\prime}(r)/\gamma(r)\right)<\infty$ otherwise there is nothing to prove. We fix $\alpha^{\prime}>\alpha>\limsup_{r\downarrow 0}\left(r\, \gamma^{\prime}(r)/\gamma(r)\right)$. By a similar argument as above there exists $r_1>0$ such that $r\mapsto \gamma(r)/r^{\alpha}$ is nonincreasing on $(0,r_1]$, and then $\lim_{r\downarrow 0}\gamma(r)/r^{\alpha}$ exists and positive. Therefore $\lim_{r\downarrow 0}\gamma(r)/r^{\alpha^{\prime}}=\infty$ and thence $\mathrm{ind}^{*}(\gamma)\leq \alpha^{\prime}$. Hence, by letting $\alpha^{\prime}\downarrow \alpha$ and $\alpha\downarrow \limsup_{r\downarrow 0}\left(r\, \gamma^{\prime}(r)/\gamma(r)\right)$, we obtain the desired inequality.
    \end{proof}
    \begin{remark}
    Notice that if in addition $\gamma$ is concave then $\limsup_{r\downarrow 0}\left(\frac{r\, \gamma^{\prime}(r)}{\gamma(r)}\right)\leq 1$.
    \end{remark}
    
   \noindent Recall that $\gamma$ is said to be a {\it regularly varying function near $0$} with index $\alpha\in (0,1]\,$ if it can be represented as 
   $$
   \gamma(x)=x^{\alpha}\, L(x),
   $$
    for all $x\in(0,x_0)$ for some $x_0>0$, where $L: (0,x_0)\rightarrow [0,\infty)$ is a slowly varying function at $0$ in the sense of Karamata, see for example \cite{Bingham et al}. Moreover such a slowly varying function can be represented as 
    \begin{align}
    L(x)=\exp\left(\eta(x) + \int_{x}^{x_0}\frac{\varepsilon(t)}{t}dt\right),
    \end{align}
    where $\eta,\varepsilon: [0,x_0)\rightarrow \mathbb{R}$, are Borel measurable and bounded functions, such that  
    $$ 
    \lim_{x\rightarrow 0}\eta(x)=\eta_0\in (0,\infty) \quad \text{ and  } \quad \lim_{x\rightarrow 0}\varepsilon(x)=0.
    $$
    For more details one can see Theorem 1.3.1 in \cite{Bingham et al}. It is known from Theorem 1.3.3 and Proposition 1.3.4 in \cite{Bingham et al} and the ensuing discussion that there exists $\widetilde{L}:(0,x_0]\rightarrow \R_+$ which is $\mathcal{C}^{\infty}$ near zero such that  $L(x) \thicksim \widetilde{L}(x)$ as $x\rightarrow 0$, and $\widetilde{L}(\cdot)$ has the following form
    \begin{equation}\label{nice rep slow var}
        \widetilde{L}(x)=\mathsf{c}_{10}\,\exp\left( \int_{x}^{x_0}\frac{\widetilde{\varepsilon}(t)}{t}dt\right),
    \end{equation}
    for some positive constant $\mathsf{c}_{10}$. Such function is called normalized slowly varying function (Kohlbecker \cite{Kohlbecker 58}), and in this case 
    \begin{equation}\label{epsilon calcul}
     \widetilde{\varepsilon}(x)=-x\, \widetilde{L}^{\prime}(x)/\widetilde{L}(x) \quad \text{for all }  \,x\in (0,x_0).
    \end{equation} 
For more properties of regularly varying functions see Seneta \cite{Seneta} or Bingham et al. \cite{Bingham et al}. 
    
\begin{remark}
     It is remarkable that Lemma \ref{bnd for lwr-uppr ind} implies that when the limit $\alpha:=\lim_{r\downarrow 0}\left(\frac{r\gamma^{\prime}(r)}{\gamma(r)}\right)$ exists, then $\mathrm{ind}_{*}(\gamma)=\mathrm{ind}^{*}(\gamma)=\alpha$. Moreover one then readily checks that if $\alpha>0$, then $\gamma(\cdot)$ is regularly varying with index $\alpha$, and in this case, $\gamma(\cdot)$ can be represented as $\gamma(x)=x^{\alpha}\, L(x)$ for all $x\in (0,x_0]$ for some $x_0\in (0,1)$, where $L(x)=\mathsf{c}_{10}\,\exp\left(\int_x^{x_0} \frac{\varepsilon(t)}{t}dt\right)$, and $\varepsilon(x)=-\frac{x\,L^{\prime}(x)}{L(x)}=\alpha-\frac{x\gamma^{\prime}(x)}{\gamma(x)}$. 
\end{remark}
\noindent The following result ensures that all regularly varying functions with indexes in $(0,1)$ satisfy \eqref{condition raisonable}.
\begin{proposition}\label{prop RVF}
Let $\gamma$ be a regularly varying function near 0, with index $\alpha \in (0,1]$. Then $\gamma$ satisfies \eqref{condition raisonable}.
\end{proposition}    
\begin{proof}
Since $\gamma$ is a regularly varying function we represent it as $\gamma(x)=x^{\alpha}\,L(x)$ for all $x\in (0,x_0)$ as discussed above. By a result of Adamović  \cite[Proposition 1.3.4]{Bingham et al}, since we are interested only in the asymptotic behavior of $\gamma$ near 0, we may assume without loss of generality that the slowly varying part $L(\cdot)$ is $\mathcal{C}^{\infty}$ and has the representation \eqref{nice rep slow var}. 
Now let 
$$
I(x):= \frac{1}{ \gamma(x)}\dint_0^{1/2}\gamma(xy)\frac{dy}{y\sqrt{\log(1/y)}}.
$$
Then we only need to show that $I(x)$ is bounded as $x$ approaches $0$. We first have 
\begin{align}
    \begin{aligned}
      I(x)=\frac{x^{\alpha}}{ \gamma(x)}\,\dint_0^{1/2}L(xy)\frac{dy}{y^{1-\alpha}\sqrt{\log(1/y)}}&\leq  \frac{\log^{-1/2}(2)\,x^{\alpha}}{\gamma(x)}\, \dint_0^{1/2}L(xy)\frac{dy}{y^{1-\alpha}}\\
    & \leq \frac{\log^{-1/2}(2)}{\gamma(x)}\, \dint_0^xL(z)\frac{dz}{z^{1-\alpha}}.
    \end{aligned}
\end{align}
It is easy to check that $\gamma^{\prime}(x)=x^{\alpha-1}\,L(x)\left(\alpha-\varepsilon(x)\right)$. Thus we may apply l'H\^opital's rule to get that
\begin{align*}
    \limsup_{x\downarrow 0}\frac{1}{ \gamma(x)}\dint_0^{1/2}\gamma(xy)\frac{dy}{y\sqrt{\log(1/y)}} &\leq  \lim_{x\downarrow 0}\frac{\log^{-1/2}(2)}{\gamma(x)}\int
_0^xL(z)z^{\alpha-1}dz\\
&=\lim_{x\downarrow 0}\frac{\log^{-1/2}(2)\,x^{\alpha-1}\,L(x)}{x^{\alpha-1}\,L(x)\left(\alpha-\varepsilon(x)\right)}=\log^{-1/2}(2)/\alpha<\infty,
\end{align*}
since $\alpha>0$. This  finishes the proof.  
\end{proof}
\noindent Here are some examples of regularly varying functions which immediately satisfy Condition \textbf{$(\mathbf{C_0})$}.
\begin{example}\label{explHe ofareVF}
\,

\begin{itemize}
    \item[i)]$\gamma_{\alpha,\beta}(r):=r^{\alpha}\log^{\beta}(1/r)$ for $\beta \in \mathbb{R}$ and $\alpha\in (0,1)$,
    \item[ii)]$\gamma_{\alpha,\beta}(x):=x^{\alpha}\,\exp\left(\log^{q}(1/x)\right)$ for $q \in (0,1)$ and $\alpha \in (0,1)$,
    \item[iii)]$\gamma_{\alpha}(x):=x^{\alpha}\, \exp\left(\frac{\log(1/x)}{\log\left(\log(1/x)\right)}\right)$ for $\alpha\in (0,1)$.
\end{itemize}
\end{example}

On the other hand, one of our goals in this paper is to study path properties for continuous Gaussian processes, satisfying Condition $\mathbf{(\Gamma)}$, i.e. relations \eqref{commensurate}, within or beyond the Hölder scale.  
If $\hbox{ind}_{*}(\gamma)>0$, it is not difficult to check that all trajectories of $X$ are $\beta$-Hölder continuous for any $\beta\in \left(0,\mathrm{ind}_{*}(\gamma)\right)$. When ${\mathrm{ind}_{*}}(\gamma)={\mathrm{ind}^{*}}(\gamma)=0$, the trajectories of $X$ are never H\"older continuous. Since all continuous Gaussian  processes must live at least in the logarithmic scale, i.e we should have $\gamma(x)=o\left( \log^{-\beta}(1/r)\right)$ for some $\beta\geq 1/2$. Thinking of this logarithmic scale as the most irregular one, there are several other regularity scales which interpolate between H\"older-continuity scale and the aforementioned logarithmic scale. This compels us to ask the following question: Is there a continuous and increasing function $\gamma$ with $\hbox{ind}_{*}(\gamma)=\hbox{ind}^{*}(\gamma)=0$ which satisfies \eqref{condition raisonable}? 

Noting that most examples of interest of function $\gamma$  with  $\hbox{ind}_{*}(\gamma)=\hbox{ind}^{*}(\gamma)=0$ are slowly varying in the sense of Karamata, for any such function $\gamma$, \cite[Proposition 1.3.4]{Bingham et al} ensures that $\gamma$ is commensurate with a $\mathcal{C}^{\infty}$ function $\gamma_0$ which satisfies $\lim_{x\downarrow 0}\frac{x\, \gamma^{\prime}_{0}(x)}{\gamma_{0}(x)}=0$. Then the following proposition addresses the aforementioned compelling question, essentially providing a negative answer.


    \begin{proposition}\label{prop suffic cond}
    Let $\gamma:[0,1] \rightarrow \mathbb{R}_+$ be a differentiable increasing function and assume that $\lim_{x\downarrow 0}x\, \gamma^{\prime}(x)/\gamma(x)=0$. Then   
\begin{align}\label{cond not hold}
    \lim_{x\downarrow 0}\left(\frac{1}{\gamma(x)}\dint_0^{1/2}\gamma(x\, y)\frac{dy}{y\sqrt{\log(1/y)}}\right)=\infty.
    \end{align}
    \end{proposition}

    \begin{proof} From Lemma \ref{bnd for lwr-uppr ind}, since $\lim_{x\downarrow 0}\frac{x\, \gamma^{\prime}(x)}{\gamma(x)}=0$ implies that $\hbox{ind}_{*}(\gamma)=\hbox{ind}^{*}(\gamma)=0$, and $\gamma(\cdot)$ is normalized regularly varying at zero, hence it can be represented as $\gamma(x)=\mathsf{c}_8\,\exp\left(\int_x^{x_0} \varepsilon(t)/tdt\right)$ where $\varepsilon(x):=-\frac{x\gamma^{\prime}(x)}{\gamma(x)}$ for some fixed $x_0\in (0,1)$. 
    Then using Fatou's Lemma we obtain
\begin{align}\label{lwr bnd integr 1}
     \liminf_{x\downarrow 0}\left(\frac{1}{\gamma(x)}\dint_0^{1/2}\gamma(x\, y)\frac{dy}{y\sqrt{\log(1/y)}}\right)&\geq \dint_0^{1/2}\lim_{x\downarrow 0}\left(\frac{\gamma( x\, y)}{\gamma( x)}\right)\frac{dy}{y\sqrt{\log(1/y)}}\nonumber\\
     &= \dint_{0}^{1/2}\exp\left(\lim_{x\downarrow 0}\int_{xy}^{x} \varepsilon(t)/tdt\right) \frac{dy}{y\sqrt{\log(1/y)}}\\
     &= \dint_{0}^{1/2} \frac{dy}{y\sqrt{\log(1/y)}}=\infty, \nonumber
 \end{align}
where, from the second to the third line, we used the facts that, for any fixed $y\in (0,1/2)$, we have
$$
\left\vert \int_{xy}^{x} \varepsilon(t)/tdt \right\vert \leq \log(1/y)\,\sup_{t\in (0,x)}\vert\varepsilon(t)\vert,
$$
for all $x\in (0,x_0)$, and that $\lim_{x\downarrow 0}\lvert \varepsilon(x)\rvert=0$. This finishes the proof.
    \end{proof}

The last result shows that condition \textbf{($\mathbf{C_{0}}$)} fails for a wide array of functions $\gamma$ with zero index. Thus condition \textbf{($\mathbf{C_{0}}$)} will not help to provide information on the upper bounds of the Hausdorff dimension of image and graph and the hitting probabilities for Gaussian processes whose  modulus of continuity is  slowly varying. We must therefore devise a weaker condition than \textbf{($\mathbf{C_{0}}$)}, satisfied by a larger class of $\gamma$'s, including slowly varying functions. First of all, for $\varepsilon>0$ we propose the following condition. 

\vspace{0.25cm}
\textbf{$\mathbf{(C_{\varepsilon})}$}: There exist three constants $\varepsilon \in (0,1)$, $\mathsf{c}_{\varepsilon}>0$ and $x_{\varepsilon}>0$, such that
\begin{align}\label{nice condition}
    \dint_0^{1/2}\gamma(xy)\frac{dy}{y\sqrt{\log(1/y)}}\leq \mathsf{c}_{\varepsilon}\, \left(\gamma(x)\right)^{1-\varepsilon}\text{ \quad for all $0<x<x_{\varepsilon}$}.
\end{align}
The following condition, denoted by \textbf{$\mathbf{(C_{0+})}$}, is weaker than \textbf{$\mathbf{(C_{0})}$} and it will be helpful to give some optimal upper bounds for the Hausdorff dimension of the image and graphe of $X$ and the hitting probabilities.

\vspace{0.25cm}
\textbf{$\mathbf{(C_{0+})}$}: For all $\varepsilon>0$ there exist two constants $\mathsf{c}_{\varepsilon}>0$ and $x_{\varepsilon}>0$, such that \eqref{nice condition} is satisfied.

\vspace{0.25cm}
The following example shows that the weaker condition \textbf{$\mathbf{(C_{0+})}$} is satisfied by a large class of functions $\gamma$ with $\mathrm{ind}_{*}(\gamma)=\mathrm{ind}^{*}(\gamma)=0$. 


\begin{example}\label{exmpl weak cond}
    Let $q\in (0,1)$ and let $\gamma_{q}$ be the function defined by $\gamma_{q}(x):=\exp\left(-\log^{q}(1/x)\right)$ for $x\in [0,1]$. Then $\gamma_{q}$ satisfies \textbf{$\mathbf{(C_{0+})}$}. 
    \end{example}
    \begin{remark} 
    Let us prove the claim in Example \ref{exmpl weak cond}. We have 
\begin{equation}\label{estim integral 2}
    \begin{aligned}
        \int_{0}^{1 / 2} \gamma_{q}(x y) \frac{d y}{y \sqrt{\log (1 / y)}}&=\int_{0}^{1 / 2} \exp\left(-\left(\log(1/x)+\log(1/y)\right)^{q}\right) \frac{d y}{y \sqrt{\log (1 / y)}}\\
        &=\int_{\log 2}^{\infty}\exp\left(-\left(\log(1/x)+z\right)^{q}\right)\frac{dz}{\sqrt{z}},
    \end{aligned}
\end{equation}
where we used the change of variable $z=\log(1/y)$. Using the fact that, for all $\mathfrak{c}\in (0,1)$ there is some $N:=N(\mathfrak{c})>0$ large enough, so that  
\begin{align}
    (1+u)^{q}\geq 1+\mathfrak{c}\,u^{q} \quad \text{for all $u\geq N$, }\label{real estimat}
\end{align}
we may fix $\mathfrak{c}\in (0,1)$, and its corresponding $N(\mathfrak{c})$. Then we break the integral in \eqref{estim integral 2} into the intervals $\,[\log(2), N \log(1/x))\,$ and $\,[N\log(1/x), +\infty)\,$ and denote them by $\mathcal{I}_1$ and $\mathcal{I}_2$, respectively. We write
\[
\left(\log(1/x)+z\right)^{q}=\log^{q}(1/x)\times\left(1+z/\log(1/x)\right)^{q}, 
\]
and we note that the second term is bounded from below by $1+\mathfrak{c}\left(\frac{z}{\log(1/x)}\right)^{q}$ when $z\geq N\log(1/x)$ due to \eqref{real estimat}, and bounded from below by $1$ when $z< N\log(1/x)$. Therefore 
\begin{align}\label{I1}
\mathcal{I}_1\leq\exp\left(-\log^{q}(1/x)\right)\, \int_{0}^{N\log(1/x)}\frac{dz}{\sqrt{z}}=2\,\gamma_{q}(x)\, \sqrt{N\,\log(1/x)}.
\end{align}
On the other hand
\begin{align}\label{I2}
    \mathcal{I}_2\leq \exp\left(-\log^{q}(1/x)\right)\,\int_{0}^{\infty}\operatorname{e}^{-\mathfrak{c}z^{q}}\frac{dz}{\sqrt{z}}=\mathsf{c}(q)\, \gamma_{q}(x).
\end{align}
Combining \eqref{I1}, \eqref{I2} and the fact $\,\sqrt{\log(1/x)}=o\left(\gamma_{q}^{-\varepsilon}(x)\right)$ for all $\varepsilon>0$, the proof of the claim in Example \ref{exmpl weak cond} is complete.
\end{remark}

As announced in the introduction, we spend some effort in this paper to study the Hausdorff dimensions of image sets and graphs, and associated hitting probabilities, for extremely irregular continuous Gaussian processes, those which satisfy Condition \textbf{$\mathbf{(C_{\varepsilon})}$} for some $\varepsilon\in (0,1)$. We use the logBm processes as a main source of examples. Proving that logBm is non-H\"older-continuous can be done ``by hand'' by employing a classical technique to establish a liminf on the gauge function in the H\"older-modulus of continuity, as is done for Brownian motion. It can also be established by invoking Fernique's  zero-one law regarding gauge functions of Gaussian processes, which states that any gauge function of the path of such a process must be a sub-Gaussian variable, and must thus have a finite expected value. This property can then be combined with the known optimality of Dudley's so-called entropy integral as an upper and lower bound for Gaussian processes with stationary increments, up to multiplicative constants.
This proof strategy must be adapted to deal with the issue that the increments of $B^{\gamma}$ are only roughly stationary in the sense of commensurability (as defined as in relations \eqref{commensurate}). The same proof structure also works to show that the process $B^{\gamma_q}$ defined using $\gamma_q$ in Example \ref{exmpl weak cond} is not H\"older-continuous, and similarly to prove that that an a.s. modulus of continuity for $B^{\gamma_q}$ is not an a.s. modulus of continuity for any logBm. 

The details of these proofs are not within the scope of this paper, and are left to the interested reader, who will find \cite{Adler 2, Marcus-Rosen, Eulalia&Viens2013} and results in the current section herein instructive. In justifying Example \ref{exmpl weak cond}, we proved that the standard deviation function $\gamma$  of $B^{\gamma_q}$ satisfies \textbf{$\mathbf{(C_{0+})}$}; the reader will easily check that the standard deviation function $\gamma$ of logBm satisfies \textbf{$\mathbf{(C_{1/2\beta})}$} but fails to satisfy \textbf{$\mathbf{(C_{\varepsilon})}$} for all $\varepsilon\in (0,1/2\beta)$. 

\subsection{Hausdorff measure, Hausdorff dimension and Riesz-Bessel capacity on $\mathbb{R}_+$ and $\R_+\times \R^d$ equipped with general metrics}
To give formula for the Hausdorff dimension of the image $X(E)$ and the graph $Gr_E(X)$ under some general conditions on $\gamma$, we must first provide appropriate notions of Hausdorff measure and Hausdorff dimension associated with a general metric $\delta$, since these will apply in particular with $\delta$ equal to the canonical metric $\delta$. 

Let $\delta: [0,1]\times [0,1]\rightarrow \mathbb{R}_+$ be a metric on $[0,1]$. For $\beta>0$ and $E \subset [0,1]$, the $\beta$-dimensional Hausdorff measure of $E$ in the metric $\delta$ is defined by 
    \begin{equation}
    \mathcal{H}_{\delta}^{\beta}(E):=\lim_{\eta \rightarrow 0}\inf \left\{\sum_{n=1}^{\infty}\left(2 r_{n}\right)^{\beta}: E \subseteq \bigcup_{n=1}^{\infty} B_{\delta}\left(r_{n}\right), r_{n} \leqslant \eta \right\}.\label{parabolic Hausdorff measure 0}
    \end{equation}
    The associated Hausdorff dimension is defined as 
    \begin{equation}\label{dim delta}
        \dim_{\delta}(E):=\sup\left\{\beta>0: \mathcal{H}_{\delta}^{\beta}(E)>0 \right \}.
    \end{equation}
The Bessel-Riesz capacity of order $\beta$ in the metric $\delta$ is defined by 
\begin{equation}
    	\mathcal{C}^{\beta}_{\delta}(E):=\left[\inf _{\nu \in \mathcal{P}(E)} \mathcal{E}_{\delta,\beta}(\nu)\right]^{-1},\label{delta capacity}
\end{equation}
    	where $\mathcal{E}_{\delta,\beta}(\nu)$ denote the $\beta$-energy of a measure $\nu \in \mathcal{P}(E)$ in the metric space ${\delta}$, defined as  
$$ 
\mathcal{E}_{\delta,\beta}(\nu):= \int_{\mathbb{R}_+} \int_{\mathbb{R}_+} \frac{\nu(d t) \nu(d s)}{(\delta(t,s))^{\beta}}.
$$
If $\delta$ is the Euclidean metric on $\mathbb{R}^n$ for some $n$ we denote the associated $\beta$-energy by $\mathcal{E}_{\operatorname{euc},\beta}(\cdot)$ and the corresponding Bessel-Riesz capacity by $\mathcal{C}_{\operatorname{euc}}^{\beta}(\cdot)$.
There exists  an alternative expression for the Hausdorff dimension given through the Bessel-Riesz capacities by
\begin{align}\label{altern dim delta}
    \dim_{\delta}(E)=\sup\left\{\beta>0: \mathcal{C}_{\delta}^{\beta}(E)>0 \right \}.
\end{align}
It is useful to understand from whence formula \eqref{altern dim delta} comes. The fact that the right hand of \eqref{altern dim delta} is a lower bound for $\dim_{\delta}(E)$ is due to the so-called energy method (see for example Theorem 4.27 in \cite{Peres&Morters}). That it is an upper bound comes from an application of Frostman's Lemma in the metric space $\left([0,1],\delta \right)$, as we now explain. 

Since capacities  are non-negative, if $\dim_{\delta}(E)=0$, then the upper bound in \eqref{altern dim delta} holds. We thus assume that $\dim_{\delta}(E)>0$. It was proven in \cite{Ho95} that, if $E$ is any subset of some general metric space $(Z,\delta)$ then we have 
\begin{align}\label{dim Frost}
     \dim_{\delta}(E)=\sup\left\{\beta \,:\, \exists r_0>0, \mathsf{c}_0>0, \text{ and } \nu \in \mathcal{P}(E) : \nu\left(B_{\delta}(z,r)\right)\leq \mathsf{c}_0\, r^{\beta} \text{ for all } r<r_0 \text{ and } z\in Z \right\}.
\end{align}
See for example Proposition $5$ and  Note $12$ in \cite{Ho95} for a good understanding of this last formulation, which we now use to prove the remaining inequality in \eqref{altern dim delta}. 
Let $\alpha \in (0,\dim_{\delta}(E))$, and fix some $\beta\in (\alpha,\dim_{\delta}(E))$. Equality \eqref{dim Frost} implies that there exists $\nu \in \mathcal{P}(E)$, $0<r_0<1$, and $0<\mathsf{c}_0<\infty$ such that 
\begin{equation}\label{Frost cond nu}
    \nu\left(B_{\delta}(z,r)\right)\leq \mathsf{c}_0\, r^{\beta} \quad \text{ for all $r<r_0$ and $z\in Z$.}
\end{equation}
For a fixed $t\in E$, since \eqref{Frost cond nu}  ensures that $\nu$ has no atom, we derive the following decomposition:
\begin{align}\label{decomp energy}
    \begin{aligned}
    \int_E\frac{\nu(ds)}{\delta(t,s)^{\alpha}}=\sum_{k=1}^{\infty}\int_{\delta(t,s)\in (2^{-k},2^{-k+1}]}\frac{\nu(ds)}{\delta(t,s)^{\alpha}}&\leq \sum_{k=1}^{\infty}2^{k\alpha}\nu\left(B_{\delta}(t,2^{-k+1})\right)\\
    &\leq \mathsf{c}_1\, \sum_{k=1 }^{\infty}2^{-k(\beta-\alpha)},
    \end{aligned}
\end{align}
with $\mathsf{c}_1=2^{\beta}\,\mathsf{c}_0$. The last sum is finite since $\alpha<\beta$, and does not depend on $t\in E$. Using the fact that $\nu$ is a probability measure, we deduce that $\mathcal{E}_{\delta,\alpha}(\nu)<+\infty$. 
which finishes the proof of the upper bound part in \eqref{altern dim delta}.

We will also need Hausdorff-dimension notions to quantify the size of the graphs of our processes as subsets of $\mathbb{R}_+\times \mathbb{R}^d$. Let $\rho_{\delta}$ be the metric defined on $\mathbb{R}_+\times \mathbb{R}^d$ via
\begin{align}\label{parabolic metric}
    \rho_{\delta}\left((s,x),(t,y)\right):=\max\{\delta(t,s),\|x-y\|\}, \quad \text{ for all } (s,x),(t,y)\in \mathbb{R}_+\times \mathbb{R}^d.
\end{align}
For $\beta>0$ and $G\subseteq \mathbb{R}_+\times\mathbb{R}^d$ be a Borel set, the $\beta$-dimensional Hausdorff measure of $G$ in the metric $\rho_{\delta}$ is defined by 
\begin{equation}
\mathcal{H}_{\rho_{\delta}}^{\beta}(G)=\lim_{\eta \rightarrow 0}\inf \left\{\sum_{n=1}^{\infty}\left(2 r_{n}\right)^{\beta}: G \subseteq \bigcup_{n=1}^{\infty} B_{\rho_{\delta}}\left(r_{n}\right), r_{n} \leqslant \eta \right\}.\label{parabolic Hausdorff measure 1}
\end{equation}
Let us also recall the so called $\beta$-Hausdorff content in the metric $\rho_{\delta}$, which is defined as follows
\begin{equation}\label{Haus content}
\mathcal{H}_{\rho_{\delta},\infty}^{\beta}\left(G\right)=\inf \left\{\sum_{i=1}^{\infty}\left|G_i\right|_{\rho_{\delta}}^\beta: G\subset \bigcup_{i=1}^{\infty}G_i\right\},
\end{equation}
where the infimum is taken over all possible covering of $G$, not merely ball coverings, and where $\lvert \cdot \rvert_{\rho_{\delta}}$ denotes the diameter in the metric $\rho_{\delta}$. The corresponding Hausdorff dimension of $G$ is defined and characterized by
\begin{equation}\label{Haus dim product 1}
    \dim_{\rho_{\delta}}(G):=\inf\{\beta\geq 0: \mathcal{H}_{\rho_{\delta}}^{\beta}(G)=0\}=\inf\{\beta\geq 0: \mathcal{H}_{\rho_{\delta},\infty}^{\beta}(G)=0\}.
\end{equation}
For the proof of the second equality above one can see Proposition 4.9 in \cite{Peres&Morters}. The Bessel-Riesz capacity of order $\alpha$ of $G$, in the metric $\rho_{\delta}$, is defined by 
\begin{equation}
    \mathcal{C}_{\rho_{\delta}}^{\alpha}(G)=
    \left[\inf _{\mu \in \mathcal{P}(E)} \int_{\mathbb{R}_+\times\mathbb{R}^d} \int_{\mathbb{R}_+\times \mathbb{R}^d} \frac{\mu(d u) \mu(d v)}{(\rho_{\delta}(u,v))^{\alpha}}\right]^{-1}.\label{rho-delta capacity}
\end{equation}
Using the same arguments \eqref{dim Frost} and \eqref{decomp energy}, used for \eqref{altern dim delta}, we can deduce the following alternative expression of $\dim_{\rho_{\delta}}(\cdot)$ in terms of Bessel-Riesz capacities:
\begin{equation}\label{altern dim rho_delta}
     \dim_{\rho_{\delta}}(G)=\sup\left\{\alpha\geq 0: \mathcal{C}_{\rho_{\delta}}^{\alpha}(G)>0 \right \}.
\end{equation}

\section{Hausdorff dimension for the rank $X(E)$ and graph $Gr_E(X)$}
\subsection{Less irregular Processes}
Let $E\subset [0,1]$ be a general Borel set. Our goal in this subsection is to give minimal conditions on $\gamma$ under which upper and lower bounds for the Hausdorff dimension of the image $X(E)$ and the graph $Gr_E(X)$ are well quantified, and are preferably explicit. When $X$ has stationary  increments and $\mathrm{ind}_{*}(\gamma)>0$, an explicit formula for the Hausdorff dimension of $X(E)$ under the Euclidean metric was provided 
by Hawkes in \cite[Theorem 2]{Hawkes}. The following lemma shows that the condition $\mathrm{ind}_{*}\left(\gamma\right)>0$ generically ensures that $\gamma$ satisfies Condition \textbf{$\mathbf{(C_{0+})}$}. We also saw in the previous section that the converse if far from true, since \textbf{$\mathbf{(C_{0+})}$} allows regularity classes with zero index.  

\begin{lemma}\label{lem check nice cond}
Let $\gamma$ be continuous, increasing, and concave near the origin. If we assume that $\mathrm{ind}_{*}(\gamma) >0 $, then $\gamma$ satisfies Condition \textbf{$\mathbf{(C_{0+})}$}.
\end{lemma}
\begin{proof}
By  a change of variable and an integration by part, we obtain that for $x\in (0,1)$ sufficiently small, we have
\begin{align}
I(x)&:=\int_{0}^{1/2}\gamma(xy)\frac{dy}{y\sqrt{\log(1/y)}}\\
&=\int_{0}^{x/2}\sqrt{\log\left(x/y\right)}d\gamma(y)-\sqrt{\log(2)}\gamma(x/2)\nonumber\\
&\leq \int_{0}^{x}\sqrt{\log\left(1/y\right)}d\gamma(y)= \dint_{0}^{\gamma(x)}\sqrt{\log\left(\frac{1}{\gamma^{-1}(u)}\right)}du\nonumber.
\end{align}
Fix an arbitrary $\alpha \in (0,\mathrm{ind}(\gamma))$, then $\gamma(x)=o(x^{\alpha})$ near zero and so $u^{1/\alpha}=o\left(\gamma^{-1}(u)\right)$ near zero also. Therefore, for any fixed $\varepsilon\in(0,1)$, there exists $\mathsf{c}_\varepsilon<\infty$ and $x_{\varepsilon}\in (0,1/2]$ such that for all $x \in (0,x_{\varepsilon}]$,
\begin{align*}
    I(x) & \leq \alpha^{-1/2}\, \int_{0}^{\gamma(x)}\sqrt{\log\left(1/u\right)}du\\
    & = \alpha^{-1/2}\left(\gamma(x)\, \sqrt{\log\left(\frac{1}{\gamma(x)}\right)}+\dint_0^{\gamma(x)}\frac{dy}{\sqrt{\log(1/y)}}\right)\\
    & \leq 2\, \alpha^{-1/2}\, \gamma(x)\, \sqrt{\log\left(\frac{1}{\gamma(x)}\right)} \\ & <\mathsf{c}_\varepsilon \left(\gamma(x)\right)^{1-\varepsilon}.
\end{align*}
Since $\varepsilon$ is arbitrarily small, the proof is complete.
\end{proof}
We relax the stationarity of increments, by assuming only that $\delta$, the canonical metric of $X$, is commensurate with $\gamma$, i.e.  $\gamma$ satisfies relations \eqref{commensurate}. Then we have the following result, which also eliminates the need for a positive index. 
\begin{theorem}\label{Hausd dim image}
Let $X:[0,1]\rightarrow \R^d$ be a continuous $d$-dimensional centered Gaussian process with i.i.d. scalar components who all share a  canonical metric $\delta$ satisfying Condition $\mathbf{(\Gamma)}$, i.e. relations \eqref{commensurate}. The following statements hold. \begin{itemize}
    \item[i)] 
For any Borel set $E\subset [0,1]$,    \begin{align}\label{dim lower bnd}
    \dim_{\operatorname{euc}}(X(E)) \geq d\wedge\dim_{\delta}(E)\quad \text{ a.s.}
    \end{align}
    and     
    \begin{align}\label{dim lwr bnd graph} 
    \dim_{\rho_{\delta}}\left(Gr_E(X)\right)\geq \dim_{\delta}(E) \quad \text{ a.s.} 
    \end{align}
\item[ii)] Assume in addition that the function $\gamma$ in Condition $\mathbf{(\Gamma)}$ satisfies 
Condition \textbf{$\mathbf{(C_{\varepsilon})}$} for some $\varepsilon\in (0,1)$. Then for any Borel set $E\subset [0,1]$,             
\begin{align}\label{dim upper bnd}
    \dim_{\delta}(E)\wedge d \leq \dim_{\operatorname{euc}}(X(E))\leq d\wedge \left(\dim_{\delta}(E)+\varepsilon\,d\right)\quad \text{ a.s.}  
\end{align}
    and     
\begin{align}\label{dim uppr bnd graph}    
    \dim_{\delta}(E) \leq \dim_{\rho_{\delta}}\left(Gr_E(X)\right)\leq \dim_{\delta}(E)+\varepsilon\,d \quad \text{ a.s.} 
\end{align}
    where $\dim_{\operatorname{euc}}(\cdot)$ denote the Hausdorff dimension associated with the Euclidean metric.
\end{itemize}
\end{theorem}
\begin{corollary}\label{joli corollaire}
 Let $X:[0,1]\rightarrow \mathbb{R}^d$ be a Gaussian process as in Theorem \ref{Hausd dim image} such that $\delta$  Condition $\mathbf{(\Gamma)}$. If $\gamma$ satisfies Condition \textbf{$\mathbf{(C_{0+})}$} then we have 
 \begin{align}
     \dim_{\operatorname{euc}}(X(E))=d\wedge \dim_{\delta}(E) \quad \text{ and } \quad \dim_{\rho_{\delta}}\left(Gr_E(X)\right)= \dim_{\delta}(E) \quad \text{ almost surely}.
 \end{align}
\end{corollary}
\noindent Before proving this Theorem \ref{Hausd dim image} we introduce some notation. Let $\mathfrak{I}=\bigcup_{n=0}^{\infty}\mathfrak{I}_n$ be the class of all $\gamma$-dyadic subintervals of $[0,1]$ such that the elements of each subclass $\mathfrak{I}_n$ are of the form 
$$
I_{j,n}:=[(j-1)\gamma^{-1}(2^{-n}), j\gamma^{-1}(2^{-n})],
$$
for $n \in \mathbb{N}$ and $1 \leq j \leq \left(\gamma^{-1}( 2^{-n})\right)^{-1}$. By using relations \eqref{commensurate} and substituting  $\delta$-balls by $\gamma$-dyadic intervals in the definition of Hausdorff measure, we obtain another family of outer measures $\{\widetilde{H}_{\delta}^{\beta}(\cdot)\,:\,\beta>0\}$. Making use of relations \eqref{commensurate} we can check that for all fixed $\beta$, the measures $\mathcal{H}_{\delta}^\beta(\cdot)$ and $\widetilde{H}_{\delta}^{\beta}(\cdot)$ are commensurate and then are equivalent. The detailed proof of this equivalence, omitted here for brevity, follows the  lines of Taylor and Watson \cite{Tay&Wats} p. 326., 
 which applies immediately due to the concavity of $\gamma$ on a neighborhood of $0$.

\begin{proof}[Proof of Theorem \ref{Hausd dim image}] We begin by proving $(i)$. 
Let $\zeta<d\wedge \dim_{\delta}(E)$, then \eqref{altern dim delta} implies that there is a probability measure $\nu$ supported on $E$ such that 
\begin{align}
  \int_E\int_E\frac{\nu(ds)\nu(dt)}{\left(\delta\left(s,t\right)\right)^{\zeta}}<\infty.
\end{align}
Let $\mu:= \nu \circ X^{-1}$ be the image of $\nu$ by the process $X$, then by transfer theorem, Fubini's theorem and scaling property we have
\begin{align}
    \begin{aligned}
    \mathbb{E}\left(\dint_{\mathbb{R}^{2d}}\frac{\mu(dx)\mu(dy)}{\|x-y\|^{\zeta}}\right)&=\dint_{E^2}\mathbb{E}\left(\frac{1}{\|X(t)-X(s)\|^{\zeta}}\right)\nu(ds)\nu(dt)\\
    &=\mathsf{c}_{1,\zeta}\, \dint_{E^2}\frac{\nu(ds)\nu(dt)}{\delta(t,s)^{\zeta}}<\infty,
    \end{aligned}
\end{align}
where $\mathsf{c}_{1,\zeta}:=\mathbb{E}\left(1/\|Z\|^{\zeta}\right)$ with $Z \sim \mathcal{N}(0,I_d)$, which is finite because $\zeta<d$. Then $\mathcal{C}_{\operatorname{euc}}^{\zeta}(X(E))>0$ \,a.s. 
Hence the classical Frostman theorem ensures that $\dim_{\operatorname{euc}}\left(X(E) \right)\geq \zeta$ \,a.s., and letting $\zeta \uparrow d\wedge \dim_{\delta}(E)$ we obtain \eqref{dim lower bnd}. Let us now prove \eqref{dim lwr bnd graph}, let $\zeta < \dim_{\delta}(E)$ be arbitrary and let $\nu$ be the probability measure such that $\mathcal{E}_{\delta,\alpha}(\nu)<\infty$. Let $\widetilde{\mu}:= \nu \circ Gr(X)^{-1}$ be the image of $\nu$ by the map $t\mapsto (t,X(t))$, then again transfer theorem, Fubini's theorem and scaling property imply that
\begin{align}
    \begin{aligned}
\mathbb{E}\left(\dint_{(\mathbb{R}_+\times\mathbb{R}^d)^2}\frac{\widetilde{\mu}(dx)\widetilde{\mu}(dy)}{\left(\rho_{\delta}((t,x),(s,y)\right)^{\zeta}}\right)
&=\dint_{E^2}\mathbb{E}\left(\frac{1}{\left(\delta(t,s)\vee\|X(t)-X(s)\|\right)^{\zeta}}\right)\nu(ds)\nu(dt)\\
    &=\mathsf{c}_{2,\zeta}\, \dint_{E^2}\frac{\nu(ds)\nu(dt)}{\delta(t,s)^{\zeta}}<\infty,
    \end{aligned}
\end{align}
where $\mathsf{c}_{2,\zeta}:=\mathbb{P}[\lVert Z \rVert \leq 1]+ \mathbb{E}\left[\lVert Z\rVert^{-\zeta}\operatorname{1}_{[\lVert Z \rVert \geq 1]}\right]$ with $Z \sim \mathcal{N}(0,I_d)$, which is finite whenever $\zeta$ is. Then $\mathcal{C}_{\rho_{\delta}}^{\zeta}\left(Gr_E(X)\right)>0$ \,a.s.
Hence \eqref{altern dim rho_delta} implies that $\dim_{\rho_{\delta}}Gr_E(X)\geq \zeta$ \,a.s. and by letting $\zeta \uparrow \dim_{\delta}(E)$ the desired lower bound \eqref{dim lwr bnd graph} follows.

Now let us prove $(ii)$, the lower bounds follow from (i), so it is sufficient to establish the upper bounds. We only prove \eqref{dim uppr bnd graph}, and the assertion in \eqref{dim upper bnd} follows from a projection argument. 
Let $\zeta>\dim_{\delta}(E)$, by definition of Hausdorff dimension we have  ${\mathcal{H}}_{\delta}^{\zeta}(E)=0$ and then $\widetilde{\mathcal{H}}_{\delta}^{\zeta}(E)=0$. Let $\eta>0$ be arbitrary, then there is a family of $\gamma$-dyadic interval $(I_k)_{k\geq 1}$ such that for every $k\geq 1$ there is $n_k\in \mathbb{N}$, $1 \leq j_k \leq \left(\gamma^{-1}( 2^{-n_k})\right)^{-1}$ and $I_k:=[(j_k-1)\,\gamma^{-1}\left(2^{-n_k}\right),j_k\,\gamma^{-1}\left(2^{-n_k}\right)]$ and we have
\begin{align}
    E\subset\bigcup_{k=1}^{\infty}I_k \quad \text{ and }\quad \sum_{k=1}^{\infty}\left|I_k\right|_{\delta
    }^{\zeta}< \eta,\label{cover Hausdorff 1}
\end{align}
where $\left|\cdot\right|_{\delta}$ denote the diameter associated to the metric $\delta$. For all fixed $n\geq 1$, let $M_n$ be the number of indices $k$ for which $n_k=n$, which is obviously finite due to right hand part of \eqref{cover Hausdorff 1}. Let us denote the corresponding $\gamma$-dyadic intervals by $I_i^{n}$ for $i=1,\ldots, M_n$.
It is not hard to check, using the commensurability condition $\mathbf{(\Gamma)}$, i.e. \eqref{commensurate}, that for all $i=1,\ldots,M_n$ we have $ \mathsf{c}_3\, 2^{-n} \leq \left|I_i^n\right|_{\delta}\leq \mathsf{c}_4\, 2^{-n}$ where the constants $\mathsf{c}_3$ and $\mathsf{c}_4$ depend on $l$ only. Then 
\begin{equation}\label{sum approximate Haus}
    \sum_{n=1}^{\infty}M_n 2^{-n\,\zeta}<\eta/\mathsf{c}_3.
\end{equation}

\noindent Let $K\subset \mathbb{R}^{d+1}$ be an arbitrary compact set, we will construct an adequate covering of $Gr_E\left(X\right)\cap K$. To simplify we suppose that $K=[0,1]^{d+1}$. For every $n\geq 1$ let $\mathfrak{C}_n$ be the collection of Euclidean dyadic subcubes of $[0,1]^d$ of side length $2^{-n}$, and for all $i=1,...,M_n$ let $\mathcal{G}_{n,i}$ be the collection of cubes $C\in \mathfrak{C}_n$ such that $X\left(I_i^{n}\right)\cap C\neq \emptyset$. 
Then we have 
\begin{equation}
    \begin{aligned}\label{inclusion intersection 0}
        Gr_E\left(X\right)\cap [0,1]^{d+1}\subseteq \bigcup_{n=1}^{\infty}\,\,\bigcup_{i=1}^{M_n}\bigcup_{\,\,C\in \mathcal{G}_{n,i}} I_{i}^n\times C.
    \end{aligned}
\end{equation}
Let $\varepsilon\in (0,1)$ such that $\gamma$ satisfies Condition \textbf{$\mathbf{(C_{\varepsilon})}$}. For all $n\geq 1$, $i\in \{1,...,M_n\}$ and $ C\in \mathfrak{C}_n$, \eqref{estim small ball 1} and \eqref{nice condition} imply that 
\begin{align}\label{chance of being in cover}
    \mathbb{P}\left\{C\in \mathcal{G}_{n,i}\right\}\leq \mathsf{c}_5 2^{-n\,(1-\varepsilon)d},
\end{align}
where $\mathsf{c}_5$ depends on $\varepsilon$ only. Combining \eqref{sum approximate Haus}, \eqref{inclusion intersection 0} and \eqref{chance of being in cover} we obtain 
\begin{align}
\begin{aligned}
\mathbb{E}\left(\mathcal{H}_{\rho_{\delta},\infty}^{\zeta+\varepsilon\,d}\left(Gr_E(X)\cap [0,1]^d \right)\right)&\leq \mathsf{c}_6\, \sum_{n=1}^{\infty}\sum_{i=1}^{M_n}\sum_{I\in \mathfrak{C}_n}2^{-n\left(\zeta+\varepsilon\,d\right)}\mathbb{P}\{C\in \mathcal{G}_{n,i}\}\\
&\leq \mathsf{c}_7 \sum_{n=1}^{\infty}M_n\operatorname{Card}(\mathfrak{C}_n) 2^{-n(d+\zeta)}\\
        &=\mathsf{c}_7\sum_{n=1}^{\infty}M_n 2^{-n\, \zeta}<\mathsf{c}_8\,\eta,\label{estim content}
\end{aligned}
\end{align}
where $\mathcal{H}_{\rho_{\delta},\infty}^{\alpha}(\cdot)$ represent the $\alpha$-Hausdorff content in the metric $\rho_{\delta}$ which is defined in \eqref{Haus content} and the constants $\mathsf{c}_6$, $\mathsf{c}_7$ and $\mathsf{c}_8$ depend on $\varepsilon$ only. Since $\eta>0$ is arbitrary we get that 
$$
\mathcal{H}_{\rho_{\delta},\infty}^{\zeta+\varepsilon\,d}\left(Gr_E(X)\cap K \right)=0 \quad \text{ almost surely},
$$
and therefore, by using \eqref{Haus dim product 1}, we have $\dim_{\rho_{\delta}}\left(Gr_{E}(X)\cap K \right)\leq \zeta+\varepsilon\,d$\, a.s. for all $K\subset \mathbb{R}_+\times \mathbb{R}^d$. Hence by the countable stability of Hausdorff dimension and by making $\varepsilon\downarrow 0$ and $\zeta\downarrow \dim_{\delta}(E)$ we get the desired upper bound in \eqref{dim uppr bnd graph}. 

Finally, the upper bound in \eqref{dim upper bnd} follows directly from the facts that Hausdorff dimension does not increase by taking projection.
\end{proof}

Here are some interesting cases that are covered by our study in this section
\begin{example} 
\,
\begin{itemize}
    \item[\textbf{i)}] \textit{Lipschitz scale}: Let $\gamma$ be defined near $0$ by 
$\gamma(r):=r\,L(r),$
where $L(\cdot)$ is a slowly varying function at $0$ with $\lim_{0+}L\left(r\right)\in (0,+\infty]$, and let $\delta$ such that relations \eqref{commensurate} (Condition $\mathbf{(\Gamma)}$) are satisfied. Then it is not difficult to show that $\dim_{\delta}(E)=\dim_{\operatorname{euc}}(E)$,
where $\dim_{\operatorname{euc}}(\cdot)$ denote the Hausdorff dimension associated to the Euclidean metric on $\mathbb{R}_+$.
\item[\textbf{ii)}]
\textit{Hölder scale}: For  $\alpha\in (0,1)$ let $\gamma$ be defined defined near $0$ by $\gamma(r)=r^{\alpha}L\left(r\right)$, where $L(\cdot)$ is a slowly varying function at $0$, and let $\delta$ satisfying $\mathbf{(\Gamma)}$. Then it can be shown easily, using the slowly varying property of $L(\cdot)$, that  $\dim_{\delta}(E)=\dim_{\operatorname{euc}}(E)/\alpha$.
\item[\textbf{iii)}]\textit{Beyond the Hölder scale}: For $q\in (0,1)$ let $\gamma$ be defined by $ \gamma_{q}(x):=\exp\left(-\log^{q}(1/x)\right)$ and $\delta$ such that \eqref{commensurate} holds.
First, note that for any Borel set $E\subset [0,1]$ such that $\dim_{\delta}(E)<\infty$, by using the fact that $r^{\alpha}=o\left(\gamma(r)\right)$ for any $\alpha>0$, one can show that $\dim_{euc}(E)=0$. Hence the Euclidean metric is not sufficient to describe the geometry of some Borel sets. 
\end{itemize}
\end{example}

\subsection{Most irregular processes (LogBm)}
Now, when $\gamma(x) = \log^{-\beta}(1/x)$\, for some $\beta > 1/2$, Condition \textbf{$\mathbf{(C_{0+})}$} fails to holds, we only have   
$$
\int_0^{1/2}\gamma(xy)\frac{dy}{y\sqrt{\log(1/y)}} \asymp \gamma(x)\, 
\sqrt{\log(1/x)} = \left(\gamma(x)\right)^{1-1/2\beta},
$$
which means that $\gamma$ satisfies Condition \textbf{$\mathbf{(C_{1/2\mathbf{\beta}})}$}, but none of Conditions \textbf{$\mathbf{(C_{\varepsilon})}$} for $\varepsilon\in (0,1/2\beta)$ are satisfied. 
On the other hand, since $\delta(t,s) \asymp \log^{-\beta}\left(\frac{1}{|t-s|}\right)$, it follows that 
\begin{equation}\label{dimlog}
\dim_{\delta}(E)=\dim_{\log}(E)/\beta,
\end{equation}
where $\dim_{\log}(\cdot)$ is the Hausdorff dimension in the metric $\delta_{\log}(t,s):=\log^{-1}(1/|t-s|)$. Therefore Theorem \ref{Hausd dim image} ensures that 
\begin{equation}\label{dim upr bnd logBM}
\frac{\dim_{\log}(E)}{\beta}\wedge d \leq \dim_{\operatorname{euc}}X(E)\leq\frac{1}{\beta}\left(\dim_{\log}(E)+\frac{d}{2}\right)\wedge d\,
\end{equation}
and 
\begin{equation}\label{dim upr bnd logBM 2}
\frac{\dim_{\log}(E)}{\beta}\leq \dim_{\rho_{\delta}} Gr_E(X)\leq \frac{1}{\beta}\left(\dim_{\log}(E)+\frac{d}{2}\right).
\end{equation}
The upper bounds above might be improved, by using an alternative covering argument based on the uniform modulus of continuity of $X$. This is what the following proposition shows  
\begin{proposition}
Let $X$ be a $d$-dimensional Gaussian process such that the canonical metric $\delta$ is commensurate with $\gamma(r)=\log^{-\beta}(1/r)$ for some $\beta>1/2$. Then almost surely
\begin{equation}\label{uppr bnd dim logBm}
\dim_{\operatorname{euc}} X(E)\leq \frac{\dim_{\log}(E)}{\beta -1/2} \wedge d \quad \textup{ and } \quad \dim_{\rho_{\delta}} Gr_E(X)\leq \frac{\dim_{\log}(E)}{\beta -1/2},
\end{equation}
for all $E\subset [0,1]$.
\end{proposition}
\begin{proof}
First, by relations
\eqref{commensurate} and the fact that $\gamma$ is increasing near the origin with $\gamma(0)=0$, we have that 
\begin{equation}\label{mod cont log scale}
    \Phi_{\gamma}(r):= \gamma(r) \sqrt{\log(1/r)}+\dint_{0}^{r}\frac{\gamma( y)}{y\sqrt{\log(1/r)}}dy,
\end{equation}
is a uniform modulus of continuity for $X$, see for example \cite[Theoerem 7.2.1 p. 304]{Marcus-Rosen}). Then there is $\Omega_0\subset\Omega$ such that $\mathbb{P}(\Omega_0)=1$ and  for all $\omega \in \Omega$ there exists  a random number $\eta_{0}(\omega)\in (0,1)$
such that 
\begin{align}\label{uniform modulu}
    \sup_{|t-s|\leq \eta}| X(t)-X(s)| \leq \mathsf{c}_1\, \Phi_{\gamma}(\eta) \quad \text{ for all $0\leq \eta<\eta_0(\omega)$},
\end{align}
where $\mathsf{c}_1$ is a positive constant. 
Since $\Phi_{\gamma}(\eta)=O\left(\log^{-(\beta-1/2)}(\eta)\right)$, then \eqref{uniform modulu} ensures that for all $0<r<\log^{-1}(\frac{1}{\eta_0(\omega)})$ \,the image of any ball $B_{\delta_{\log}}(t,r)$ by $X(\cdot,\omega)$ has a diameter smaller than $\mathsf{c}_1 (2r)^{\beta-1/2}$. Let $\omega\in \Omega_0$ be fixed and let $E\subseteq [0,1]$ such that $\dim_{\delta}(E)<\infty$. Then for any $\xi>\dim_{\log}(E)$, there is a covering of $E$ by balls $\left\{B_{\delta_{\log}}\left(t_i,r_i\right):i\geq 1\right\}$ such that $\sum_{i=1}^{\infty}(2r_i)^{\xi}\leq \varepsilon$ for some $\varepsilon$ arbitrarily small which we choose such that $\varepsilon^{1/\xi}\leq 2\log^{-1}(\frac{1}{\eta_0(\omega)})$, then $Gr_E(X)$ is covered by the family $\left\{B_{\delta_{\log}}(t_i,r_i)\times X\left(B_{\delta_{\log}}(t_i,r_i)\right):i\geq 1\right\}$ and we have  
\begin{align*}
\mathcal{H}_{\rho_{\delta},\infty}^{\xi/(\beta-1/2)}\left(Gr_E(X)\right)&\leq \sum_{i=1}^{\infty}\left(\left|B_{\delta_{\log}}(t_i,r_i)\times X\left(B_{\delta_{\log}}(t_i,r_i)\right)\right|_{\rho_{\delta}}\right)^{\xi/(\beta-1/2)}\\
&\leq \mathsf{c}_2\,\sum_{i=1}^{\infty}(2r_i)^{\xi}\leq \mathsf{c}_2\, \varepsilon.
\end{align*}
Since $\varepsilon$ is arbitrarily small we get $\mathcal{H}_{\rho_{\delta},\infty}^{\xi/(\beta-1/2)}\left(Gr_E(X)\right)=0$ and consequently $\dim_{\rho_{\delta}}Gr_E(X)\leq \xi/(\beta-1/2)$. By letting $\xi \downarrow \dim_{\log}(E)$ the proof is complete.
\end{proof}
\begin{remark}\label{Req unif uppr bnd}
\hspace{1in}
\begin{itemize}
    \item[i)] The upper bounds in \eqref{uppr bnd dim logBm} are uniform in the sense that the negligible set does not depend on $E$. The covering method used in this proof can be adapted to show that, under the following stronger condition $(\widetilde{\textbf{C}}_{0+})$: ``$\Phi_{\gamma}(r)=o\left(\gamma^{1-\varepsilon}(r)\right)$  near zero for all $\varepsilon>0$ small enough", the upper bounds $\dim_{\delta}(E)\wedge d$ and $\dim_{\delta}(E)$ are uniform for $X(E)$ and $Gr_E(X)$, respectively.
    \item[ii)] Let $E\subset [0,1]$ such that $0<\dim_{\log}(E)<\infty$ then by combining \eqref{dim lower bnd}, \eqref{dim upr bnd logBM} and \eqref{uppr bnd dim logBm} we obtain 
\begin{align*}
        \frac{\dim_{\log}(E)}{\beta}\wedge d\leq \dim_{\operatorname{euc}}(X(E))\leq \frac{\dim_{\log}(E)}{\beta-1/2}\wedge d\quad \text{ a.s. }
\end{align*}
This is due to the fact $\frac{1}{\beta}\left(\dim_{\log}(E)+\frac{d}{2}\right)\geq \frac{\dim_{\log}(E)}{\beta-1/2}\wedge d$. Hence the upper bound nearly agrees with the lower bound near the upper (less irregular) end of the logarithmic scale, i.e. for large $\beta$.
\end{itemize}
\end{remark}
Since the previous methods lead to different upper and lower bounds for Hausdorff dimensions of the image and the graph in the logarithmic scale, it is interesting to ask the following question: Are the random variables $\dim_{\rho_{\delta}} \left(Gr_E(X)\right)$ and $\dim_{\operatorname{euc}} X(E)$ constant almost surely in this logarithmic scale? The main goal of the remaining part of this section is to answer this question.
The key probabilistic idea is to use the Karhunen-Lo\`eve expansion of the process $X$ so that we can show that the random variables $\dim_{\rho_{\delta}}\left(Gr_E(X)\right)$ and $\dim_{\operatorname{euc}} X(E)$ are measurable with respect to a tail sigma-field, and therefore by the zero-one law of Kolmogorov they should be almost surely constants. Let us first recall the Karhunen-Lo\`eve expansion, which says that $X$ has the following $\mathcal{L}^2$-representation, see for example \cite[Theorem 3.7 p. 70 \,\,and (3.25) p. 76]{Adler 2} :
\begin{equation}\label{karh-loe decomp}
    X(t)=\sum_{i=1}^{\infty}\lambda_i^{1/2}\,\xi_i\,\psi_i(t),
\end{equation}
where $(\xi_i)_{i\geq 1}$ is an i.i.d. sequence of  $N(0,I_d)$ standard Gaussian vectors, and $(\lambda_i)_{i\geq 1}$ and $(\psi_i)_{i\geq 1}$ are respectively eigenvalues and eigenvectors of the covariance operator of $Q_X$, defined on $L^2([0,1])$ by
$$
(Q_X\psi)(t)=\dint_0^1Q(s,t)\psi(s)ds,
$$
where $Q(s,t):=\mathbb{E}\left[X_0(s)X_0(t)\right]$ is the covariance function of each component of $X$. 
It is easy to see from \eqref{karh-loe decomp} that the canonical metric $\delta$ has the following representation
\begin{equation}\label{delta}
    \delta(s,t)=\left(\sum_{i=1}^{\infty}\lambda_i(\psi_i(t)-\psi_i(s))^{2}\right)^{1/2}.
\end{equation}
In addition, this formula shows that every eigenfunction $\psi_i$ is continuous, since all eigenfunctions share $\delta$ as a modulus of continuity up to a multiplicative constant, i.e. $|\psi_i(t)-\psi_i(s)| \leq \lambda_i^{-1/2}\delta(s,t)$.

\begin{theorem}\label{0-1 law for the graph}
Let $\{X(t): t\in [0,1] \}$ be a $d$-dimensional continuous Gaussian process as defined in \eqref{d iid copies}, satisfying the commensurability condition $\mathbf{(\Gamma)}$, i.e. relations \eqref{commensurate}, such that 
\begin{equation}\label{suffic cont cond}
    \lim_{r\rightarrow 0}\gamma(r)\,{\log}^{1/2}(1/r)=0.
\end{equation}
Then for all Borel set $E\subset (0,1)$ there is a non-random constant $\mathbf{C}(E)\in [0,+\infty]$ such that   
\begin{align}\label{dim graph constant}
    \dim_{\rho_{\delta}}\left( Gr_E(X)\right)=\mathbf{C}(E) \quad a.s.
\end{align}
\end{theorem}
\noindent The following deterministic lemma is a key to prove Theorem \ref{0-1 law for the graph}.
\begin{lemma}\label{effect of Lipshitz drift}
Let $f:\,[a,b]\rightarrow \mathbb{R}^d$ be a Borel measurable function and $g:\,[a,b]\rightarrow \mathbb{R}^d$ be a Lipschitz in the metric $\delta$, i.e. 
\begin{equation}\label{lipsh drift}
    \lVert g(t)-g(s) \rVert \leq \mathsf{C}_g\, \delta(t,s) \quad \text{ for all } s,t\in [a,b],
\end{equation}
for some positive constant $\mathsf{C}_g$. Then for all Borel set $E\subseteq [a,b]$ we have 
\begin{equation}\label{No effect of lipschitz drift}
\dim_{\rho_{\delta}}\left(Gr_E(f+g)\right)=
\dim_{\rho_{\delta}}\left(Gr_E(f)\right).
\end{equation}
\begin{proof}
    Let $\alpha:=\dim_{\rho_{\delta}}\left(Gr_E(f)\right).$  Then\, $\mathcal{H}_{\rho_{\delta}}^{\alpha+\varepsilon}\left(Gr_E(f)\right)=0$ for all $\varepsilon>0$. Therefore we fix $\varepsilon>0$ and $\eta>0$ to be arbitrary so that there exists a cover $\left(B_{\delta}(t_i,r_i)\times B(x_i,r_i)\right)_{j\geq 1}$ of $Gr_E(f)$ such that 
    \begin{align}
\sum_{j=1}^{\infty}r_i^{\alpha+\varepsilon}<\eta.
    \end{align}
From this last cover of $Gr_E(f)$ we will construct another cover of $Gr_E(f+g)$. Indeed, by using $\eqref{lipsh drift}$ if $t\in B_{\delta}(t_i, r_i)$ for some $i\geq 1$, then 
\begin{equation}\label{cover g}
    \lVert g(t)-g(t_i)\rVert \leq \mathsf{C}_g\, r_i.
\end{equation}
Now let $i\geq 1$ such that $t\in B_{\delta}(t_i,r_i)$ and $f(t)\in B(x_i,r_i)$, we then deduce from this and from \eqref{cover g} that $(f+g)(t)\in B\left(\widetilde{x}_i,\widetilde{r}_i\right)$ where $\widetilde{x}_i:=x_i+g(t_i)$ and $\widetilde{r}_i:=(1+\mathsf{C}_g)r_i$. Therefore the collection of balls $\left(B_{\delta}\left(t_i,\widetilde{r}_i\right)\times B(\widetilde{x}_i,\widetilde{r}_i)\right)$ is a cover of $Gr_E(f+g)$ and we have 
$$
\mathcal{H}^{\alpha+\varepsilon}_{\rho_{\delta},\infty}\left(Gr_E(f+g)\right)\leq (1+\mathsf{C}_g)^{\alpha+\varepsilon}\sum_{j=1}^{\infty}r_i^{\alpha+\varepsilon}\leq (1+\mathsf{C}_g)^{\alpha+\varepsilon}                                                     \, \eta.
$$
Since $\eta>0$ is arbitrary, this shows that $\mathcal{H}^{\alpha+\varepsilon}_{\rho_{\delta},\infty}\left(Gr_E(f+g)\right)=0$ for all $\varepsilon>0$. Hence \eqref{Haus dim product 1} ensures that
\begin{equation}\label{1st ineq dim}
\dim_{\rho_{\delta}}\left(Gr_E(f+g)\right)\leq \alpha=\dim_{\rho_{\delta}}\left(Gr_E(f)\right).
\end{equation}
The other inequality follows from \eqref{1st ineq dim} with $\widetilde{f}:=f+g$ and $\widetilde{g}:=-g$.
\end{proof}
\end{lemma}

\begin{proof}[Proof of Theorem \ref{0-1 law for the graph}]
First let us note that \eqref{suffic cont cond} implies that $X$ has a continuous version, then by using \cite[Theorem  3.8]{Adler 2} the series in \eqref{karh-loe decomp} converge uniformly on $[0,1]$ a.s.2, thus it is a concrete version of $X$.  Considering this version, we define for all $n\geq 1$ the finite and infinite parts of $X$, denoted by $X_{1,n}$ and $X_{n,\infty}$ as follows
$$
X_{1,n}(t):=\sum_{i=1}^n \lambda_i^{1/2}\,\xi_i\,\psi_i(t) \quad \text{ and } \quad X_{n,\infty}(t):= X(t)-X_{1,n}(t) \quad \text{for all } t\in [0,1].
$$
Then we have 
\begin{equation}\label{lipschitz part}
    \left\lVert X_{1,n}(t)-X_{1,n}(s) \right\rVert \leq \left(\sum_{i=1}^n \lvert \xi_i\rvert \right)\, \sup_{1 \leq i\leq n} \lambda_i^{1/2} \lvert  \psi_i(t) -\psi_i(s)\rvert \leq \left(\sum_{i=1}^n \lvert \xi_i\rvert \right)\,\delta(t,s),
\end{equation}
for all $s,t\in [0,1]$, almost surely, where we used \eqref{delta} in the last inequality.
We fix $E\subset [0,1]$ to be a Borel set. By making use of \eqref{lipschitz part}, Lemma \ref{effect of Lipshitz drift} applies for almost every $\omega$; specifically, for fixed $n$, this is the set of $\omega$'s such that $\sum_{i=1}^n \lvert \xi_i\rvert $ is finite. Lemma \ref{effect of Lipshitz drift} thus ensures that, by countable intersection, almost surely,
$$
\dim_{\rho_{\delta}}Gr_E(X)=\dim_{\rho_{\delta}}Gr_E(X_{n,\infty})\quad \quad \text{ for all $n\geq 1$}.
$$
This shows that the random variable $\dim_{\rho_{\delta}}\left(Gr_E(X)\right)$ is measurable with respect to the tail $\sigma$-algebra $\bigcap_{n=1}^{\infty}\sigma\left(\{\xi_i, i\geq n+1\}\right)$. Hence the Kolmogorov's 0-1 law ensures that this random variable is constant almost surely.
\end{proof}

\begin{remark}\label{finite number}
    The proof of Theorem \ref{0-1 law for the graph} relies on the fact that the dimension of the graph of the process $X$ is in the tail sigma-algebra of a sequence of i.i.d random variables. But the Karhunen-Lo\`eve expansion of $X$ may have only finitely many non-zero terms, making that tail sigma-algebra property arguably artificial. Still, the proof's argument carries through, though the result of the theorem can be obtained more directly. Indeed, if $\lambda_i=0$ for all $i$ greater than some fixed $n_0$, and assuming that the eigenfunctions $\psi_i$ are differentiable for all $i\leq n_0$ and at least one of them satisfies the fact that: $\lvert \psi_i(t)-\psi_i(s)\rvert \geq \mathsf{c}_i\, \lvert t-s\rvert$ for $s,t\in J$ for some $i\leq n_0$ and some  interval $J\subset [0,1]$. Then the canonical metric of the process is commensurate with the Euclidean metric on $J$, and Corollary \ref{joli corollaire} proves that, for all $E\subset J$, the Hausdorff dimension of $Gr_E(X)$ equals the usual Hausdorff dimension of $E$. More generally, still assuming that all $\lambda_i=0$ for all $i$ greater than some fixed $n_0$, but without assuming that the eigenfunctions $\psi_i$ are differentiable, by applying Lemma \ref{effect of Lipshitz drift} with $f=X_{n_0,\infty} \equiv 0$ and $g=X_{1,n_0}=X$, we get $\dim_{\rho_\delta}(Gr_E(X))=\dim_{\rho_{\delta}}(Gr_E(0))=\dim_{\rho_\delta}(E\times {0})=\dim_{\delta}(E)$.
\end{remark}

Due to the complex structure of the image compared to the graph, the previous methodology, which is based on a covering argument and Hausdorff measures techniques to show that $\dim_{\rho_{\delta}} Gr_E(X)$ is measurable with respect to a tail sigma-field, is difficult to be applied to the image case, which pushes us to seek other methods. To prove a similar result for $\dim_{\operatorname{euc}} X(E)$ we will proceed differently, trying to use the Karhunen-Loève expansion again combined with a potential theoretical approach to be able to prove that $\dim_{\operatorname{euc}} X(E)$ is measurable with respect to the tail sigma-field associated with the sequence of Gaussian random variables appearing in the Karhunen-Lo\`eve expansion. 
\begin{theorem}\label{0-1 law for image}
Under the same conditions of Theorem \ref{0-1 law for the graph} we have for all Borel set $E\subset [0,1]$ there exists a non-random constant $\mathbf{c}(E)\in [0,d]$ such that 
\begin{equation}\label{No effect of lipschitz drift}
\dim_{\operatorname{euc}}\left(X(E)\right)=\mathbf{c}(E)
\quad a.s.
\end{equation}
\end{theorem}

\begin{remark}
    Just as in Theorem \ref{0-1 law for the graph}, the proof of Theorem \ref{0-1 law for image} also seems to use the Kolmogorov 0-1 law artificially when there is only a finite number of nonzero Karhunen-Lo\`eve eigenvalues $\lambda_i$. Yet the same arguments as in Remark \ref{finite number} lead to a direct proof that the dimension of the image is non-random, and in fact, $\dim_{\operatorname{\operatorname{euc}}}(X(E))=\dim_{\delta}(E) \wedge d$.
\end{remark}
\begin{remark}
    We believe that the situation in the previous remark can never occur if condition $(\mathbf{C}_{0+})$ does not hold. We know of two classes of examples where no such situation can be constructed because all processes that violate  condition $(\mathbf{C}_{0+})$ in those classes have infinitely many non-zero Karhunen-Lo\`eve eigenvalues. Recall the Volterra processes in \eqref{volt repr}. Then we can prove that every eigenfunction $\psi_i$ of such a process is $\alpha$-H\"older-continuous on $[0,1]$ for any $0<\alpha<1$. The details are left to the reader. For such a process, if its Karhunen-Lo\`eve expansion had only finitely many non-zero terms, then the process would also be $\alpha$-H\"older-continuous, almost surely, which would imply, using the lower-bound side of the commensurability condition $\mathbf{(\Gamma)}$ in \eqref{commensurate}, that its standard deviation function $\gamma$ has a positive lower index, and thus that condition $(\mathbf{C}_{0+})$ holds because of Lemma \ref{lem check nice cond}; again details are omitted.  We also leave it to the reader to check that, in the case of  processes with stationary increments, the same argument via H\"older-continuity holds. Thus, for both Volterra processes and processes with stationary increments satisfying condition $\mathbf{(\Gamma)}$, we can prove by contrapositive that if condition $(\mathbf{C}_{0+})$ is violated, then the Karhunen-Lo\`eve expansion had infinitely many non-zero terms.         
    
\end{remark}

In order to prove Theorem \ref{0-1 law for image} we need some preliminaries. First we start by a classical result, whose proof is an application of Hahn-Banach theorem, see for example Theorem 1.20 p. 17 in \cite{Mattila}.     
\begin{lemma}\label{lem image measure}
Let $(E,\rho)$ be a compact metric space and $f\,:\,E\, \rightarrow \mathbb{R}^d$ be a continuous function. Then for any probability measure $\mu$ on $f(E)$ there exists a probability measure $\nu$ on $E$  such that $\mu=\nu\circ f^{-1}$.
\end{lemma}

Recall that the Karhunen-Loève expansion provides a concrete continuous version of the Gaussian process $X$, that all its eigenfunctions are continuous, and that using the notation $X_{1,n}$ and $X_{n,\infty}$ defined in the proof of Theorem \ref{0-1 law for the graph}, the function $X_{1,n}$, as a finite (random) linear combination of eigenfunctions, is continuous, and therefore, $X_{n,\infty}$ is continuous as a difference of two continuous processes. All these statements are to be understood almost surely. Let us denote by $\mathbb{Q}_{1,n}$ and  $\mathbb{Q}_{n,\infty}$ their distributions on the space of continuous functions, and by  $\delta_{1,n}$ and $\delta_{n,\infty}$ their associated canonical metrics respectively. The expression \eqref{delta} then immediately implies 
\begin{equation}
    \delta_{1,n}^2(s,t)=\sum_{i=1}^{n}\lambda_i\left(\psi_i(t)-\psi_i(s)\right)^2\quad \text{ and }\quad \delta_{n,\infty}^2(s,t)=\sum_{i=n+1}^{\infty}\lambda_i\left(\psi_i(t)-\psi_i(s)\right)^2,
\end{equation}
and these two processes are independent by construction. Therefore we have the equality in distribution $\left(X,\mathbb{P} \right)   \,\overset{d}{=}\,\left(X_{1,n}+X_{n,\infty},\mathbb{Q}_{1,n}\otimes \mathbb{Q}_{n,\infty} \right)$. For convenience, we denote by $\Omega_{1,n}$ and  $\Omega_{n,\infty}$ two copies of the space of continuous functions; the measures $\mathbb{Q}_{1,n}$ and $ \mathbb{Q}_{n,\infty}$ are defined on these two spaces. We may also choose to define the law of $X$ on the set of continuous functions $\Omega=\Omega_{1,n}\times \Omega_{n,\infty}$, and for $\omega \in \Omega$ the paths $X_{1,n}(\omega)$ and $X_{n,\infty}(\omega)$ can be understood using the obvious projection. 
\begin{proof}[Proof of Theorem \ref{0-1 law for image}]For all $n\geq 1$ and all Borel set $E\subset [0,1]$ we denote by $K_{n}(\cdot)$ the following random kernel
\begin{align}
    K_{n}(s,t,\omega):=\left(\delta_{1,n}\left(s,t\right)\vee \lVert X_{n,\infty}(s,\omega)-X_{n,\infty}(t,\omega)\rVert\right)^{-1}\quad \text{ for all $s,t \in [0,1]$ and $\omega \in \Omega$}.
\end{align}
Let $\nu$ be a probability measure on $E$. Denote by $\zeta_{n}\left(E,\cdot\right)$ the random variable defined as follow 
\begin{align}\label{random dimension}
    \zeta_n\left(E\right):= \sup\left\{ \zeta>0\,:\,\inf_{\nu\in \mathcal{P}(E)}\int_E\int_E \left[K_{n}(s,t,\cdot)\right]^{\zeta}\nu(ds)\nu(dt)<\infty\right\}.
\end{align}
We will show that for any fixed integer $n\geq 1$ and for all Borel set $E\subset [0,1]$ we have 
\begin{align}\label{value dim image}
\dim_{\operatorname{euc}} X(E)=\zeta_n(E)\wedge d\quad \text{\,  almost surely.} 
\end{align} 
Since the integers are countable, \eqref{value dim image} holds almost surely for all $n\geq 1$ simultaneously. In particular, almost surely, $\zeta_n(E)\wedge d$ does not depend on $n$.

 Indeed, let $n \geq 1$ be fixed and $E\subseteq [0,1]$ be a Borel set, we will first prove that $\dim_{\operatorname{euc}} X(E)\leq \zeta_n(E)\wedge d$\, a.s. 
 Let $\omega \in \Omega_n:=\left\{\, \max_{i\leq n}\lVert \xi_i\rVert<\infty\,\right\}$, and assume that $\zeta_n(E)(\omega)<d$ otherwise there is nothing to prove. Then \eqref{random dimension} implies that for all $\zeta>\zeta_n(E)(\omega)$ we have
 \begin{align}\label{infinite energy}
     \int_E\int_E \left[K_{n}(s,t,\omega)\right]^{\zeta}\nu(ds)\nu(dt)=\infty \quad \text{for all $\nu\in \mathcal{P}(E)$}.
 \end{align}
 
\noindent On the other hand, we note that for all $s,t\in [0,1]$ we have  
\begin{align}\label{trg ineq}
     \lVert X(t,\omega)-X(s,\omega)\rVert & \leq \lVert X_{1,n}(t,\omega)-X_{1,n}(s,\omega)\rVert + \lVert X_{n,\infty}(t,\omega)-X_{n,\infty}(s,\omega)\rVert \nonumber \\
     &\leq \left(\max_{i\leq n}\left\lVert \xi_i(\omega)\right\rVert\ \right)\delta_{1,n}\left(t,s\right) + \left\lVert X_{n,\infty}(t,\omega)-X_{n,\infty}(s,\omega)\right\rVert \nonumber \\ 
     & \leq \left(\max_{i\leq n}\lVert \xi_i(\omega)\rVert+1\right)\,\left[\delta_{1,n}\left(t,s\right)\vee \left\lVert X_{n,\infty}(t,\omega)-X_{n,\infty}(s,\omega)\right\rVert \right]\\ \nonumber
    & = \left(\max_{i\leq n}\lVert \xi_i(\omega)\rVert+1\right)\,\left[K_n(s,t,\omega)\right]^{-1}.\nonumber 
\end{align}
Thus by \eqref{infinite energy} and \eqref{trg ineq} we infer that
\begin{align}\label{all energies explose}
    \int_E\int_E \frac{\nu(ds)\nu(dt)}{\lVert X(t,\omega)-X(s,\omega)\rVert^{\zeta}}=\infty \quad \text{ for all $\nu\in \mathcal{P}(E)$}.
\end{align}
Using Lemma \ref{lem image measure}, any probability measure $\mu$ on $X(E,\omega)$ may be written as $\mu=\nu\circ X^{-1}(\cdot,\omega)$ for some $\nu \in \mathcal{P}(E)$, so using this fact as well as \eqref{all energies explose} we obtain $\mathcal{C}^{\zeta}_{\operatorname{euc}}\left(X(E,\omega)\right)=0$ and then by \eqref{altern dim delta} we have $\dim_{\operatorname{euc}} X(E,\omega)\leq \zeta$. Letting $\zeta \downarrow \zeta_n(E)(\omega)$ we get $\dim_{\operatorname{euc}} X(E,\omega)\leq \zeta_n(E)(\omega)$. Since $\mathbb{P}(\Omega_n)=1$, the desired upper bound hold almost surely for fixed $n$, and then as we mentioned, for all $n$ simultaneously.

We will now show that $\dim X(E)\geq \zeta_n(E)\wedge d$\, a.s. First, we remark that the random variable $\zeta_n(E)$ is measurable with respect to $\sigma(\left\{ \xi_i\,:\, i\geq n+1\,\right\})$ and therefore it is independent from $X_{1,n}$. Let $n\in \mathbb{N}$ and $\omega_2\in \Omega_{n,\infty}$ be fixed, we assume that $\zeta_n(E)(\omega_2)>0$ otherwise there is nothing to prove. Let $ 0<\zeta<\zeta_n(E)(\omega_2)\wedge d$ be arbitrary, then there exists a probability measure $\nu_{\omega_2}\in \mathcal{P}(E)$ such that 
\begin{align}\label{finite graph energy}
    \int_E\int_E \left[K_{n}(s,t,\omega_2)\right]^{\zeta}\nu_{\omega_2}(ds)\nu_{\omega_2}(dt)<\infty.
\end{align}
Now for any $\omega_1\in \Omega_{1,n}$ we consider the random probability measure $\mu_{\omega_1,\omega_2}$ defined on $X(E)$ via 
$$
\mu_{\omega_1,\omega_2}(F):=\nu_{\omega_2}\left(\left\{s\in E\,:\, X\left(t,(\omega_1,\omega_2)\right)\in\, F\right\}\right) \quad \text{ for all $F\subset X(E)$}.
$$
Our aim is to show that 
\begin{align}\label{finite random energy}
 \mathcal{E}_{euc,\zeta}\left(\mu_{\omega_1,\omega_2}\right)<\infty   \quad \text{ for \,$\,\mathbb{Q}_{1,n}$-almost all $\omega_1\in\Omega_{1,n}$.}
\end{align}
In fact, for $\omega_2\in \Omega_{n,\infty}$ being fixed, taking expectation with respect to $\mathbb{Q}_{1,n}(d\omega_1)$ and using a transfer theorem and Fubini's theorem we obtain that 
\begin{align}\label{expect energy image}
\begin{aligned}
    \mathbb{E}_{\mathbb{Q}_{1,n}}\left(\mathcal{E}_{euc,\zeta}\left(\mu_{\cdot,\omega_2}\right)\right)=\dint_E\dint_E \underset{:=I_{n,\zeta}(s,t,\,\omega_2)}{\underbrace{\mathbb{E}_{\mathbb{Q}_{1,n}}\left(\frac{1}{\lVert X_{1,n}(t)-X_{1,n}(s)+X_{n,\infty}(t,\omega_2)-X_{n,\infty}(s,\omega_{2}) \rVert^{\zeta} }\right)}}\nu_{\omega_2}(ds)\nu_{\omega_2}(dt).
\end{aligned}
\end{align}
In order to prove \eqref{finite random energy} we only need to show that
\begin{align}\label{compar random kernels}
    I_{n,\zeta}(t,s,\omega_2)\leq \mathsf{c}_0 \left[K_n(s,t,\omega_2)\right]^{\zeta} \quad \text{for all $s,t\in E$},
\end{align}
where $\mathsf{c}_0$ is a positive constant. Let $s,t\in E$, if $K_n(s,t,\omega_2)=\infty$ the above inequality is obvious. So we assume that $K_n(s,t,\omega_2)<\infty$. Then for simplicity we let 
$$ 
\mathsf{u}:=\delta_{1,n}(s,t) \quad \text{and}\quad \mathsf{v}(\omega_2):=X_{n,\infty}(t,\omega_2)-X_{n,\infty}(s,\omega_2).
$$
Then using the Gaussian scaling property and the independence between $X_{1,n}$ and $X_{n,\infty}$ we have 
\begin{align}\label{scal-indep}
    I_{n,\zeta}(s,t,\omega_2)=\mathbb{E}_{\mathbb{Q}_{1,n}}\left(\frac{1}{\lVert \mathsf{u}\,Z+\mathsf{v}(\omega_2)\rVert^{\zeta}}\right)=\dint_{\mathbb{R}^d}\frac{1}{\lVert \mathsf{u}\,x+\mathsf{v}(\omega_2)\rVert^{\zeta}}\frac{e^{-\frac{\lVert x \rVert^2 }{2}}}{(2\pi)^{d/2} } dx,
\end{align}
where $Z$ is a standard Gaussian vector $N(0,I_d)$.  There are four possible cases: (i) $\mathsf{u}=0<\lVert \mathsf{v}(\omega_2)\rVert$, (ii) $\lVert \mathsf{v}(\omega_2)\rVert=0<\mathsf{u}$, (iii) $0< \lVert \mathsf{v}(\omega_2)\rVert\leq \mathsf{u}$ and (iv) $0<\mathsf{u}\leq \lVert \mathsf{v}(\omega_2)\rVert $. Since $\zeta<d$, the inequality \eqref{compar random kernels} is trivial in the first two cases, let us then prove it only in the cases (iii) and (iv). First, for $\mathsf{w}:=\mathsf{v}(\omega_2)/\mathsf{u}$\, let $J(\mathsf{w})$ be defined as
\begin{align}\label{int appear in krnl}    J\left(\mathsf{w}\right):=\dint_{\mathbb{R}^d}\frac{1}{\lVert \,x+\mathsf{w}(\omega_2)\rVert^{\zeta}}\frac{e^{-\frac{\lVert x \rVert^2 }{2}}}{(2\pi)^{d/2} } dx.
\end{align}
One can remark that $I_n(s,t)=\,u^{-\zeta}\,J(\mathsf{w})$. When $0< \lVert \mathsf{v}(\omega_2)\rVert\leq \mathsf{u}$, using the fact that the functions $x\mapsto e^{-\lVert x\rVert^2/2}$ and $\,x\mapsto \lVert x \rVert^{-\zeta}$ have the same monotony as functions of $\lVert x\rVert$, then for all $\mathsf{w}\in \R^d$ we have 

\begin{equation}\label{same monotony}
    \dint_{\mathbb{R}^d}(e^{-\lVert x+\mathsf{w}\rVert^2/2}-e^{-\lVert x\rVert^2/2})(\lVert x+\mathsf{w} \rVert^{-\zeta}-\lVert x \rVert^{-\zeta})dx\geq 0.
\end{equation}
Hence using a change of variables we obtain 
\begin{align}\label{estim krnl 1st case}
\begin{aligned}
J\left(\mathsf{w}\right)\leq 2\dint_{\mathbb{R}^d}\frac{1}{\lVert \,x\rVert^{\zeta}}\frac{e^{-\frac{\lVert x \rVert^2 }{2}}}{(2\pi)^{d/2} } dx=:\mathsf{c}_{1,\zeta},
\end{aligned}
\end{align}
where $\mathsf{c}_1=\mathsf{c}_{1,\zeta}=2(2\pi)^{-d/2}\,\int_{\mathbb{R}}r^{d-\zeta-1}\,e^{-r^{2}/2}dr<\infty$ since $\zeta<d$. Then multiplying $J(\mathsf{w})$ by $\mathsf{u}^{-\zeta}$ and using the upper bound \eqref{estim krnl 1st case}, we get 
\begin{align}\label{estim krnl 1st case}
   I_{n,\zeta}(s,t,\omega_2)\leq \mathsf{c}_1\, \mathsf{u}^{-\zeta}=\mathsf{c}_1\, [K(s,t,\omega_2)]^{\zeta}.
\end{align}
This gives the desired inequality in the case (iii). On the other hand, when $0<\mathsf{u} `< \lVert \mathsf{v}(\omega_2)\rVert $ we upper bound the integral $J(\mathsf{w})$
\begin{align}\label{estim app integ 2nd case}
\begin{aligned}
J(\mathsf{w})&=(2\pi)^{-d/2}\,\left(\dint_{\lVert x+\mathsf{w}\rVert \geq \lVert \mathsf{w}\rVert/2 } \frac{1}{\lVert \,x+\mathsf{w}\rVert^{\zeta}}e^{-\frac{\lVert x \rVert^2 }{2}} dx +\dint_{\lVert x+\mathsf{w}\rVert < \lVert \mathsf{w}\rVert/2 } \frac{1}{\lVert \,x+\mathsf{w}\rVert^{\zeta}}e^{-\frac{\lVert x \rVert^2 }{2}} dx \right)\\
&\leq (2\pi)^{-d/2}\,\left( \lVert \mathsf{w}\rVert^{-\zeta}\dint_{\mathbb{R}^d}e^{-\lVert x\rVert^2/2}dx+e^{-\lVert \mathsf{w}\rVert^2/8}\dint_{\lVert x+\mathsf{w}\rVert <\lVert \mathsf{w}\rVert/2}\frac{dx}{\lVert x+\mathsf{w}\rVert^{\zeta}}\right)\\
&\leq \mathsf{c}_2\, \left(\lVert \mathsf{w}\rVert^{-\zeta}+e^{-\lVert \mathsf{w}\rVert^2/8}\,\lVert \mathsf{w}\rVert^{d-\zeta} \right)\\
&\leq \mathsf{c}_3\,\lVert \mathsf{w}\rVert^{-\zeta}, 
\end{aligned}
\end{align}
where, in the first inequality, the bound of the second term follows from the fact that $\lVert x \rVert \geq \lVert \mathsf{w} \rVert/2$, the second and third inequalities follow from passing to polar coordinates and using the facts that $\zeta<d$ and that $\underset{r\in \mathbb{R}}{\sup}\,r^d\, e^{-r^2/2}<\infty$. Thus multiplying $J(\mathsf{w})$ by $\mathsf{u}^{-\zeta}$  and using the upper bound \eqref{estim app integ 2nd case} we obtain
\begin{align}\label{estim krnl 2nd case}
    I_n(s,t,\omega_2)\leq \mathsf{c}_3\, \lVert \mathsf{v}(\omega_2)\rVert^{-\zeta}=\mathsf{c}_3\, [K(s,t,\omega_2)]^{\zeta},
\end{align}
which finishes the proof in the case $(iv)$. 

Now using \eqref{finite graph energy}, \eqref{expect energy image} and \eqref{compar random kernels} we obtain that $\mathbb
{E}_{\mathbb{Q}_{1,n}}\left(\mathcal{E}_{euc,\zeta}\left(\mu_{\cdot,\omega_2}\right)\right)<\infty$. Therefore $\mathcal{E}_{euc,\zeta}\left(\mu_{\omega_1,\omega_2}\right)<\infty$ for $\mathbb{Q}_{1,n}$-almost all $\omega_1\in \Omega_{1,n}$, which implies that $\dim_{\operatorname{euc}} X(E,(\omega_1,\omega_2)) \geq \zeta$ \,for $\mathbb{Q}_{1,n}$-almost all $\omega_1\in \Omega_{1,n}$ and for all $\zeta<d \wedge \zeta_n(E)(\omega_2)$. Hence by letting $\zeta\uparrow d \wedge \zeta_n(E)(\omega_2)$ we get that
\begin{align}\label{final lwr bnd}
\dim_{\operatorname{euc}} X(E,(\omega_1,\omega_2)) \geq d\wedge \zeta_n(E)(\omega_2) \quad \text{for $\mathbb{Q}_{1,n}$-almost all $\omega_1\in \Omega_{1,n}$.}
\end{align} 
Accordingly, since $\omega_2\in \Omega_{n,\infty}$ is arbitrarily chosen, then using Fubini's theorem and \eqref{final lwr bnd} we obtain that 
\begin{align}
\begin{aligned}
 \mathbb{P}\left[ \dim_{\operatorname{euc}}X(E)\right.&\left.\geq d\wedge \zeta_n(E)\right]\\
 &=\mathbb{Q}_{1,n}\otimes \mathbb{Q}_{n,\infty}\left\{(\omega_1,\omega_2)\,:\, \dim_{\operatorname{euc}}X(E,(\omega_1,\omega_2))\geq d\wedge \zeta_n(E)(\omega_2)\right\} \\
 &=\dint_{\Omega_{n,\infty}} \mathbb{Q}_{1,n}\left[\omega_1\in \Omega_{1,n} \,:\, \dim_{\operatorname{euc}}X(E,(\omega_1,\omega_2))\geq d\wedge \zeta_n(E)(\omega_2)\right]\mathbb{Q}_{n,\infty}(d\omega_2)\\
 &=1.
\end{aligned}
\end{align}
Hence the proof of \eqref{value dim image} is complete.  
 
Now, since $\zeta_n(E)\wedge d$ does not depend on $n$, and since for all $n\geq 1$ we have $\zeta_n(E)$ measurable with respect to $\sigma\left(\left\{\xi_i\,:\, i\geq n+1\right\}\right)$, then $\dim_{\operatorname{euc}} X(E)$ is measurable with respect to the tail sigma-field of $(\xi_i)_{i\geq 1}$ and hence by the 0-1 law of Kolmogorov, it is constant almost surely, which finishes the proof.
\end{proof}
\begin{remark}
     The previous theorems \ref{0-1 law for the graph} and \ref{0-1 law for image} only use condition \eqref{suffic cont cond} which is sufficient for the mere existence of a continuous modification for $X$. Moreover, it was shown in Theorem \ref{Hausd dim image} under Condition \textbf{$\mathbf{(C_{0+})}$} that the constant $\mathbf{c}(E)$ and $\mathbf{C}(E)$ are nothing but $\dim_{\delta}(E) \wedge d$ and $\dim_{\delta}(E)$, respectively. But even if Condition \textbf{$\mathbf{(C_{0+})}$} fails, Theorems \ref{0-1 law for the graph} and \ref{0-1 law for image} show that the Hausdorff dimension of the image and graph are almost surely constants, and this is valid for the entire class of continuous Gaussian processes, including logBm and other extremely irregular continuous processes.

\end{remark}

\section{Criteria on hitting probabilities}
In this section we develop criteria for hitting probabilities of a Gaussian process $X$ where, as before, its canonical metric $\delta$ satisfies the commensurability condition $\mathbf{(\Gamma)}$. The concavity Hypothesis \ref{Hyp2} for the standard deviation function $\gamma$ will also be generically required. We also assume that $\gamma$ satisfies Condition \textbf{$\mathbf{(C_{0})}$}, or merely \textbf{$\mathbf{(C_{\varepsilon})}$}. Under these mild conditions, we will establish lower bounds for the probability that $X$ will hit a set $F$ from a set $E$, namely $\mathbb{P}\left\{X(E)\cap F\neq \varnothing\right\}$, in terms of capacities of $E\times F$, and upper bounds on that hitting probability in terms of Hausdorff measures of $E\times F$. Our conditions are general enough to apply to large classes of Gaussian processes within and beyond the H\"older scale. In the first subsection below, we present the main results of this section, which provide estimates under both Conditions \textbf{$\mathbf{(C_{0})}$} and \textbf{$\mathbf{(C_{\varepsilon})}$} for fixed $\varepsilon \in (0,1)$. These results suggest that a critical dimension can be identified under \textbf{$\mathbf{(C_{0+})}$}, i.e. for those processes which satisfy \textbf{$\mathbf{(C_{\varepsilon})}$} for every $\varepsilon$. This is the topic of the second subsection, wherein we show that in the critical dimension case, under \textbf{$\mathbf{(C_{0+})}$}, the hitting probability's positivity cannot be decided merely based on dimensions. In the third subsection, we investigate the so-called co-dimension of the image set $X(E)$, and we show in particular that it has an explicit expression under a mild regularity condition on the set $E$.
\subsection{General hitting probability estimates}\label{1st subsec of 4 sec}
Recall the metric $\rho_{\delta}$ on the product space, defined in \eqref{parabolic metric}. Our general result is the following.
\begin{theorem}\label{Hitting proba}
Let $X$ be a $d$-dimensional Gaussian process with i.i.d. components satisfying the commensurability condition $\mathbf{(\Gamma)}$. Let $0<a<b<\infty$ and $M>0$, and let $E\subset [a,b]$ and $F\subset [-M,M]^{d}$ be two Borel sets.  With the notation and conditions in Section 2, the following holds.   
\begin{itemize}
    \item[i)] If Hypothesis \ref{Hyp2} is satisfied, then there exists a constant $\mathsf{c}_1>0$ depending only on $a,b, M$ and the law of $X$, such that    
\begin{equation}\label{lower bnd hitting}
    \mathsf{c}_1\, \mathcal{C}_{\rho_{\delta}}^{d}(E\times F)\leq \mathbb{P}\left\{X(E)\cap F\neq \emptyset\right\}.
    \end{equation}
    \item[ii)] If Condition \textbf{$\mathbf{(C_{0})}$} is satisfied, then there exists a constant $\mathsf{c}_2>0$ also depending only on $a,b,M$, and the law of $X$, such that  \begin{equation}\label{upper bnd hitting}
        \mathbb{P}\left\{X(E)\cap F\neq \emptyset\right\}  \leq \mathsf{c}_2\, \mathcal{H}_{\rho_{\delta}}^{d}\left(E\times F\right).
    \end{equation}
    \item[iii)] If Condition \textbf{$\mathbf{(C_{\varepsilon})}$} is satisfied for some $\varepsilon\in (0,1)$, then there exists a constant $\mathsf{c}_{\varepsilon,3}>0$ depending on $a,b,M, \varepsilon$, and the law of $X$, such that 
\begin{align}\label{weak uppr bnd hitting}
    \mathbb{P}\left\{X(E)\cap F\neq \emptyset \right\}\leq \mathsf{c}_{3,\varepsilon}\, \mathcal{H}_{\rho_{\delta}}^{d(1-\varepsilon)}(E\times F).
\end{align}

\end{itemize}
\end{theorem}

\begin{proof}
We begin by proving the lower bound in \eqref{lower bnd hitting}. 
Assume that $\mathcal{C}_{\rho_{\delta}}^{d}(E\times F)>0$ otherwise there is nothing to prove. This implies the existence of a probability measure $\mu\in \mathcal{P}(E\times F)$ such that 
\begin{equation}\label{bound energy}
\mathcal{E}_{\rho_{\delta},d}(\mu):=\int_{\mathbb{R}_+\times\mathbb{R}^d} \int_{\mathbb{R}_+\times \mathbb{R}^d} \frac{\mu(d u) \mu(d v)}{(\rho_{\delta}(u,v))^d} \leq \frac{2}{\mathcal{C}_{\rho_{\delta}}^{d}(E\times F)}.
\end{equation}
Consider the sequence of random measures $(m_n)_{n\geq 1}$ on $E\times F$ defined as 
\begin{align*}
    \begin{aligned}
        m_n(dt dx)&=(2 \pi n)^{d / 2} \exp \left(-\frac{n\|X(t)-x\|^{2}}{2}\right) \mu(d t d x)\\      &=\int_{\mathbb{R}^{d}} \exp \left(-\frac{\|\xi\|^{2}}{2 n}+i\langle\xi, X(t)-x\rangle\right) d \xi\,  \mu(dt d x).
    \end{aligned}
\end{align*}
Denote the total mass of $m_n$ by $\|m_n\|=m_n(E\times F)$. Let us first verify the following claim on the moments of $\|m_n\|$:
\begin{equation}\label{claim 1}    \mathbb{E}\left(\left\|m_{n}\right\|\right) \geq \mathsf{c}_{1}, \quad \text { and } \quad \mathbb{E}\left(\left\|m_{n}\right\|^{2}\right) \leq \mathsf{c}_{2} \mathcal{E}_{\rho_{\delta},d}(\mu),
\end{equation}
where the constants $\mathsf{c}_1$ and $\mathsf{c}_2$ are independent of $n$ and $\mu$.

First, we have 
\begin{equation}
\begin{aligned} \mathbb{E}\left(\|m_n\|\right) &=\int_{E\times F} \int_{\mathbb{R}^{d}} \exp \left(-\frac{\lVert \xi\rVert^{2}}{2}\left(\frac{1}{n}+\gamma^{2}(t)\right)-i\langle \xi, x\rangle\right) d \xi \mu(d t d x) \\
& \geq \int_{E\times F} \frac{(2 \pi)^{d / 2}}{\left(1+\gamma^{2}(t)\right)^{d / 2}} \exp \left(-\frac{\lVert x\rVert^{2}}{2 \gamma^{2}(t)}\right) \mu( d t d x) \\
& \geq \frac{(2 \pi)^{d / 2}}{\left(1+\gamma^{2}(b)^{d / 2}\right.} \exp \left(-\frac{d M^{2}}{2 \gamma^{2}(a)}\right)\int_{E\times F} \mu(dt dx)=: \mathsf{c}_{1},
\end{aligned}
\end{equation}
This proves the first inequality in \eqref{claim 1}. We have also 
\begin{align}\label{2nd moment}
    \begin{aligned}
\mathbb{E}\left(\left\|m_{n}\right\|^{2}\right)=\int_{(E\times F)^2}  \int_{\mathbb{R}^{2 d}} &e^{-i(\langle\xi, x\rangle+\langle\eta, y\rangle)} \,
\times \exp \left(-\frac{1}{2}(\xi, \eta)\Gamma_n(t,s)(\xi, \eta)^{T}\right) d \xi \, d \eta\,  \mu(dt d x) \mu(d s d y),
\end{aligned}
\end{align}
where $\Gamma_n(t,s)=\left(n^{-1} I_{2 d}+\operatorname{Cov}(X(s), X(t))\right)$,  where $I_{2d}$ denotes the $2d\times 2d$ identity matrix, and  where $\operatorname{Cov}(X(s), X(t))$ is the $2d$-covariance matrix of $\left(X(s),X(t)\right)$.   
Now let $\varepsilon>0$ so that \eqref{two pts LND} is satisfied for all $s,t \in [a,b]$ such that $|t-s|<\varepsilon$. Using the same lines as Step 1 and Step 2 of the proof of Theorem 2.5 in \cite{Eulalia&Viens2013} we obtain that 
$$
\mathbb{E}\left(\lVert m_n\rVert^2\right)\leq J_1+ J_2,
$$
where 
\begin{align*}
    &J_1:=\dint_{(E\times F)^{2}\cap D(\varepsilon)} \frac{(2 \pi)^d}{\left(\sqrt{\operatorname{det}\left(\Phi_n(s, t)\right)}\right)^d} \exp \left(-\frac{c_2}{2} \frac{\lVert x-y\rVert^2}{\operatorname{det}\left(\Phi_n(s, t)\right)}\right)\mu(dtdx)\mu(dsdy)\\
    &J_2:=\dint_{(E\times F)^{2}\setminus D(\varepsilon)} \frac{(2 \pi)^d}{\left(\sqrt{\operatorname{det}\left(\Phi_n(s, t)\right)}\right)^d}\,\mu(dtdx)\mu(dsdy),
\end{align*}
where $D(\varepsilon):=\left\{\left((t,x),(s,y)\right): |t-s|<\varepsilon\right\}$ and $\Phi_n(s,t):=n^{-1}I_2+ \mathrm{Cov}(X_0(s),X_0(t))$.

First we bound  $J_2$. Observe that
\begin{equation}\label{lwr bnd determinant}
    \operatorname{det}\left(\Phi_n(s, t)\right) \geq \mathbb{E}(X_0^2(s)) \mathbb{E}(X_0^2(t))-\left(\mathbb{E}X_0(t)X_0(s)\right)^2=:h(s,t).
\end{equation}
By the Cauchy-Schwartz inequality, the function $(s, t) \mapsto h(s,t)$ is nonnegative, and since $\gamma(r)=0 \Leftrightarrow r=0$, this function is strictly positive and continuous away from the diagonal $\{s=t\}$. Therefore, for all $s, t \in[a, b]$ with $|t-s|>\varepsilon$, $\operatorname{det}\left(\Phi_n(s, t)\right) \geq \mathsf{c}_3$, where $\mathsf{c}_3$ is a positive constant depending on $[a,b]$. Hence
\begin{align}
    J_2&\leq (2\pi/\mathsf{c}_3^{1/2})^d\,\dint_{(E\times F)^2\setminus D(\varepsilon)}\mu(dtdx)\mu(dsdy)\nonumber\\
    &\leq (2\pi/\mathsf{c}_3^{1/2})^d\, \sup_{(u,v)\in (E\times F)^2}\left(\rho_{\delta}\left(u,v\right)\right)^d\dint_{(E\times F)^2}\frac{\mu(du)\mu(dv)}{\rho_{\delta}\left((t,x),(s,y)\right)^d}=\mathsf{c}_4\,\mathcal{E}_{\rho_{\delta},d}(\mu).
\end{align}
Let us now bound $J_1$. If $((t,x),(s,y))\in D(\varepsilon)$ then \eqref{lwr bnd determinant} and Lemma \ref{lem two pts LND} ensures that for some constant $\mathsf{c}_5>0$ 
$$ 
\operatorname{det}\left(\Phi_n(s, t)\right)\geq \mathsf{c}_5\,\gamma^2(a)\, \delta^2(s,t).
$$
Observe that if\,  $\operatorname{det}\left(\Phi_n(s, t)\right)<\lVert x-y\rVert^2$, using the fact that $\sup_{x\in \mathbb{R}}x^{d/2}e^{-c\,x}<\infty$, then 
$$
\frac{(2 \pi)^d}{\left({\left.\operatorname{det}\left(\Phi_n(s, t)\right)\right)}\right)^{d/2}} \exp \left(-\frac{\mathsf{c}_3}{2} \frac{\lVert x-y\rVert^2}{\operatorname{det}\left(\Phi_n(s, t)\right)}\right) \leq \frac{\mathsf{c}_6}{\lVert x-y \rVert^d}.
$$
On the other hand, when $\operatorname{det}\left(\Phi_n(s, t)\right)\geq \lVert x-y\rVert^2$ we get 
$$
\frac{(2 \pi)^d}{\left({\left.\operatorname{det}\left(\Phi_n(s, t)\right)\right)}\right)^{d/2}} \exp \left(-\frac{\mathsf{c}_3}{2} \frac{\lVert x-y\rVert^2}{\operatorname{det}\left(\Phi_n(s, t)\right)}\right) \leq \frac{(2\pi)^d}{\mathsf{c}_5^{d/2}\, \gamma^d(a)\delta(s,t)^d}.
$$
Therefore we conclude that 
\begin{align}\label{bound I1}
    J_1\leq \mathsf{c}_7\, \dint_{(E\times F)^2}\frac{\mu(dtdx)\mu(dsdy)}{\left(\max\{\delta(s,t), \lVert x-y\rVert\}\right)^d}=\mathsf{c}_7\, \mathcal{E}_{\rho_{\delta},d}(\mu),
\end{align}
for some constant $\mathsf{c}_7$. The proof of our moment estimates in claim \eqref{claim 1} is complete.

Now, using these moment estimates in \eqref{claim 1} and the Paley–Zygmund inequality (c.f. Kahane \cite{Kahane}, p.8), one can check that $\{m_n, n\geq 1\}$ has a subsequence that converges weakly to a finite random measure $m_{\infty}$ supported on the set $\{(s,x)\in E \times F : X(s)=x\}$, which is positive on an event of positive probability and also satisfying the moment estimates of \eqref{claim 1}. Therefore, using again the Paley-Zygmund inequality, we conclude that
$$
\mathbb{P}\left\{X(E) \cap F \neq \varnothing\right\} \geq \mathbb{P}\left\{\|m_{\infty}\|>0\right\} \geq \frac{\mathbb{E}(\|m_{\infty}\|)^{2}}{\mathbb{E}\left(\|m_{\infty}\|^{2}\right)} \geq \frac{\mathsf{c}_{1}^{2}}{\mathsf{c}_{2} \mathcal{E}_{\rho_{\delta},d}(\mu)}.
$$
By definition of capacity, this finishes the proof of \eqref{lower bnd hitting}.


For the upper bound in \eqref{upper bnd hitting}, we use a simple covering argument. We choose an arbitrary constant $\zeta>\mathcal{H}_{\rho_{\delta}}^d(E\times F)$. Then there is a covering of $E\times F$ by balls $\{B_{\rho_{\delta}}((t_i,x_i),r_i), i\geq 1\}$ in $\left(\mathbb{R}_+\times\mathbb{R}^d,\rho_{\delta}\right)$ with small radii $r_i$, such that 
\begin{equation}\label{cover product}
    E\times F\subseteq \bigcup_{i=1}^{\infty}B_{\rho_{\delta}}((t_i,x_i),r_i)\quad  \text{with }\quad  \sum_{i=1}^{\infty}(2r_i)^{d}\leq \zeta.
\end{equation}
It follows that
		\begin{align}\label{event covering}
		\left\{ X(E)\cap F\neq\emptyset\right\}  
        &  = {\bigcup_{i=1}^{\infty}\left\{\,X\left( B_{\delta}(t_i,r_i)\right) \cap  B(x_i,r_i)\neq \varnothing \right\}} \nonumber\\
		& \subseteq\bigcup_{i=1}^{\infty}\left\{ \inf _{ t\in B_{\delta}(t_i,r_i)}\|X(t)-x_i\| \leqslant r_i\right\}. 
		\end{align}
Since Condition \eqref{condition raisonable} is satisfied, using Corollary \ref{cor estim small ball} and \eqref{event covering} we obtain 
\begin{align}\label{final estim proba}
\mathbb{P}\left\{ X(E)\cap F\neq\emptyset\right\} &\leq \sum_{i=1}^{\infty}\mathbb{P
}\left\{ \inf _{ t\in B_{\delta}(t_i,r_i)}\|X(t)-x_i\| \leqslant r_i\right\}
\nonumber \\
&\leq \mathsf{c}_8\, \sum_{i=1}^{\infty} (2r_i)^{d}\leq \mathsf{c}_8\, \zeta.
\end{align}
Let $\zeta\downarrow \mathcal{H}_{\rho_{\delta}}^d(E\times F)$, the upper bound in \eqref{upper bnd hitting} follows. 

For  the upper bound in \eqref{weak uppr bnd hitting}, first note that condition \eqref{nice condition} ensures that 
\begin{align}\label{proba inf weak bnd}
      \mathbb{P}\left\{\inf _{ t\in B_{\delta}(t,r)}\|X(t)-x\| \leqslant r\right\}\leq \mathsf{c}_9\, r^{d(1-\varepsilon)} \quad\text{ for all $0<r<r_0$ and $x\in [-M,M]^d$ }
\end{align}
where $r_0$ and $\mathsf{c}_9$ are two positive constants. Hence the proof of \eqref{weak uppr bnd hitting} follows from the same argument as in \eqref{event covering}, \eqref{cover product} and \eqref{final estim proba}, and by using \eqref{proba inf weak bnd} instead of Corollary \ref{cor estim small ball}.
\end{proof}

The following corollary suggests that $\dim_{\rho_{\delta}}(E\times F)=d$ is a critical dimension for computing hitting probabilities. 

\begin{corollary}\label{polarity in terms of dim of ExF}
Let $E,F$ be two bounded Borel sets in $\mathbb{R}_+$ and $\mathbb{R}^d$ respectively. Under Hypothesis \ref{Hyp2} and Condition \textbf{$\mathbf{(C_{0+})}$} we have
\begin{align}\label{hitt in terms of dim}
    \mathbb{P}\left\{X(E)\cap F\neq \varnothing \right\}\left\{\begin{array}{ll}
>0 & \text { if } \dim_{\rho_{\delta}}(E\times F)>d \\
=0 & \text { if } \dim_{\rho_{\delta}}(E\times F)<d
\end{array}\right..
\end{align}
\end{corollary}

We explore this criticality in the next subsection, using general sets $E$ and processes $X$.

\subsection{Hitting probabilities: undecidability in the critical dimension case}

We now show that the critical dimension case, $\dim_{\rho_{\delta}}(E\times F)=d$, is undecidable, for a large class of functions $\gamma$ satisfying \textbf{$\mathbf{(C_{0+})}$}, in the following sense: there exist compact sets $E_1,E_2 \subset [0,1]$ and $F_1,F_2\subset [M,M]^d$ such that $\dim_{\rho_{\delta}}(E_1\times F_1)=\dim_{\rho_{\delta}}(E_2\times F_2)=d$ and
\begin{align}\label{critical case}
    \mathbb{P}\left\{X(E_1)\cap F_2 \neq \varnothing \right\}>0 \quad \text{ and } \quad \mathbb{P}\left\{X(E_2)\cap F_2\neq \varnothing \right\}=0.
\end{align}
We start with providing some lower bounds and upper bounds on $\mathbb{P}\left\{X(E)\cap F\neq \emptyset \right\}$ when $E$ satisfies the Ahlfors-David regularity in the metric $\delta$. This will be the key to prove \eqref{critical case}. 
First, we recall the definition of an Ahlfors-David regular set.
\begin{definition}\label{Ahlf-Dav regular}
 Let $(X,\rho)$ be a bounded metric space, let $\alpha>0$, and let $G\subset X$. We say that $G$ is $\alpha$-Ahlfors-David regular if there exists a Borel probability measure $\mu$ on $G$ and a positive constant $\mathsf{c}_{0}$ such that 
\begin{equation}\label{Ahlf-Dav regular cond}
	\mathsf{c}_{0}^{-1}\,r^{\alpha}\leq\mu\left(B_{\rho}\left(a,r\right)\right)\leq\mathsf{c}_{0}\,r^{\alpha}\,\,\text{ for all \ensuremath{a\in G}, and all \, \ensuremath{0<r\leq1}}.
\end{equation}
\end{definition}
To best represent the delicate size of our hitting probabilities of interest, we find it necessary to introduce a finer concept of regularity for our standard deviation function $\gamma$, using slowly-varying modulation. Let $\ell:(0,\infty)\rightarrow \mathbb{R}_+$ be a slowly varying function at $0$, such that $\lim_{y\rightarrow 0}\ell(y)=c\in (0,+\infty]$. We denote the following condition $(\mathbf{C}_{\ell})$, 

$(\mathbf{C}_{\ell})$:
There exist two constants $\mathsf{c}_1>0$ and $x_0\in (0,1)$ such that 
\begin{align}\label{condition raisonable avec ell}
\int_{0}^{1 / 2} \gamma(x y) \frac{d y}{y \sqrt{\log (1 / y)}} \leq \mathsf{c}_1\, \gamma(x)\ell\left(\gamma(x)\right)\quad \text{ for all $x\in [0,x_0]$}.
\end{align}
\begin{remark}
\hspace*{1in}
\begin{itemize}
    \item[i)] This condition $(\mathbf{C}_{\ell})$ is slightly stronger than $(\mathbf{C}_{0+})$,  and weaker than $(\mathbf{C}_{0})$ when $\lim_{y\rightarrow 0}\ell(y)=+\infty$. Moreover it is satisfied by a large class of functions $\gamma$ with zero index of interest to us, including the example $\gamma(x)=\exp\left(-\log^{q}(1/x)\right)$ with $q\in (0,1)$.
    \item[ii)] When $\lim_{y\rightarrow 0}\ell(y)<+\infty$, the conditions $(\mathbf{C}_{0})$ and $(\mathbf{C}_{\ell})$ are equivalent. 
    \item[iii)] The case of  $\lim_{y\rightarrow 0}\ell(y)=0$ does not occur. Indeed, one can show that, up to a multiplicative constant, $\gamma(x)$ is a lower bound of the integral in Condition $(\mathbf{C}_{\ell})$.
\end{itemize} 
\end{remark}
This modulated condition $(\mathbf{C}_{\ell})$ is naturally accompanied by the more general notion of Hausdorff measure with a gauge function other than the power function, which we will also need. For a metric space $(X,\rho)$ and a function $\varphi:\mathbb{R}_+\rightarrow \mathbb{R}_+$, right-continuous and increasing near zero with $\lim_{0+}\varphi=0$, 
and $G\subseteq X$ be a Borel set, the $\varphi$-Hausdorff measure of $G$ in the metric $\rho$ is defined by 
\begin{equation}\label{phi-Hausd meas}
\mathcal{H}_{\rho}^{\varphi}(G)=\lim_{\eta \rightarrow 0}\inf \left\{\sum_{n=1}^{\infty}\varphi\left(2 r_{n}\right): G \subseteq \bigcup_{n=1}^{\infty} B_{\rho}\left(r_{n}\right),\,\, r_{n} \leqslant \eta \right\}.
\end{equation}

The same reasoning as in the proof of Theorem \ref{Hitting proba} leads to an upper bound more accurate than \eqref{weak uppr bnd hitting}, under the condition $(\mathbf{C}_{\ell})$. The proof of the following theorem is thus left to the interested reader. 
\begin{theorem}
    Let $0<a<b<\infty$ and $M>0$, and let $E\subset [a,b]$ and $F\subset [-M,M]^{d}$ be two Borel sets. If $\gamma$ satisfies the hypothesis \ref{Hyp2} and the condition $(\mathbf{C_{\ell}})$, then 
\begin{equation}\label{lwr-uppr bnd hitt C_l}
\mathsf{c}^{-1}_2\, \mathcal{C}^d_{\rho_{\delta}}(E\times F)\leq \mathbb{P}\left\{X(E)\cap F\neq \emptyset\right\}\leq \mathsf{c}_2\,\mathcal{H}_{\rho_{\delta}}^{\varphi_d}(E\times F),
\end{equation}
where $\varphi_d(x):=x^d\,\ell^{d}(x).$
\end{theorem}
If $E$ is an $\alpha$-Ahlfors-David regular set in the metric $\delta$, the hitting probability estimates \eqref{lwr-uppr bnd hitt C_l} take a more specific form. Namely the lower and upper bounds are given, respectively, in terms of the Bessel-Riesz capacity of $F$ and the Hausdorff measure of $F$ in the Euclidean metric, the latter still being relative to the $\ell$-modulated power function. However, when $\alpha$ reaches the critical dimension $d$, the capacity lower bound requires the use of a logarithmic metric. To be specific, we have the following proposition, whose proof, based on the previous theorem, requires a bit of care, and is therefore included below. 

\begin{proposition}\label{bnds in trm of E resp. F} 
Let $X$ be a $d$-dimensional Gaussian process such that its standard deviation function $\gamma$ satisfies Condition $\mathbf{(\Gamma)}$, Hypothesis \ref{Hyp2} and Condition \textbf{$\mathbf{(C_{\ell})}$}. Let $0<a<b<\infty$ and $M>0$. Also let $E\subset [a,b]$ be a $\alpha$-Ahlfors-David regular set in the metric $\delta$ for some $0<\alpha \leq d$. Then for all $0<M\leq 1$ and  $F\subset [-M,M]^d$ the following two alternatives hold, depending on whether $\alpha$ equals the critical dimension $d$.
 \begin{itemize}
     \item[i-1)] If $\alpha<d$ and $\gamma$ satisfies Condition \textbf{$\mathbf{(C_{0})}$}  then
\begin{equation}\label{lwr-uppr bnd(F)}         \mathsf{c}_3^{-1}\,\mathcal{C}_{\operatorname{euc}}^{d-\alpha}\left(F\right) \leq \mathbb{P}\left\{ X(E)\cap F\neq \emptyset \right\}\leq\, \mathsf{c}_3\, \mathcal{H}_{\operatorname{euc}}^{{d-\alpha}}(F),
     \end{equation} 
     
     \item[i-2)] If $\alpha<d$ and $\gamma$ satisfies Condition \textbf{$\mathbf{(C_{\ell})}$} for some $\ell$ given such that $\lim_{y\rightarrow 0}\ell(y)=+\infty$, then we have
\begin{equation}\label{lwr-uppr bnd(F)}         \mathsf{c}_3^{-1}\,\mathcal{C}_{\operatorname{euc}}^{d-\alpha}\left(F\right) \leq \mathbb{P}\left\{ X(E)\cap F\neq \emptyset \right\}\leq\, \mathsf{c}_3\, \mathcal{H}_{\operatorname{euc}}^{\varphi_{d-\alpha}}(F),
     \end{equation}
where $\varphi_{d-\alpha}(x):=x^{d-\alpha}\,\ell^{d}(x)$ and $\mathsf{c}_3$ is a positive constant depends on $a,\,b,\,M$ and $\alpha$ only.
   
   \item[ii)] If $\alpha=d$ then
   \begin{equation}\label{hitting proba when alpha=d}
\mathsf{c}_4\,\mathcal{C}_{\delta_{\log}}^{1}\left(F\right)\leq \mathbb{P}\left\{ X(E)\cap F\neq \emptyset \right\}
   \end{equation}
where the metric $\delta_{\log}(\cdot)$ is defined on $[-M,M]^d$ by \,$\delta_{\log}(x,y):=-\log^{-1}(\lVert x-y \rVert)$.
 \end{itemize}
%
%
\end{proposition}

\begin{remark}
    In the case $\alpha=d$, the upper bound in terms of the Hausdorff measure, under either Condition $(\mathbf{C}_{0})$ or Condition $(\mathbf{C}_{\ell})$ with $\lim_{y\rightarrow 0}\ell(y)=+\infty$, is not informative. Indeed, under $(\mathbf{C}_{0})$ the Hausdorff measure is a discrete measure, implying that the upper bound is typically too large to be informative,  and under Condition $(\mathbf{C}_{\ell})$ the Hausdorff measure is infinite for any nonempty set $F$.
\end{remark}
\begin{proof}
Using the bounds in \eqref{lwr-uppr bnd hitt C_l}, to prove (i) it will be sufficient to show that  
\begin{align}\label{bnd cap-haus 1}
    \mathsf{c}_5^{-1}\mathcal{C}_{\operatorname{euc}}^{d-\alpha}(F)\leq \mathcal{C}_{\rho_{\delta}}^{d}(E\times F) \quad \text{ and }\quad \mathcal{H}_{\rho_{\delta}}^{\varphi_d}(E\times F)\leq \mathsf{c}_5\, \mathcal{H}_{\operatorname{euc}}^{\varphi_{d-\alpha}}(F),
\end{align}
respectively. Indeed for the capacities inequality, since $E$ is $\alpha$-Ahlfors-David regular in the metric $\delta$, then by using \cite[Proposition 2.5]{Erraoui Hakiki 2}, with $G_1=E$, $G_2=F$, $\rho_{1}=\delta$, $\rho_{2}=\lVert \cdot \rVert$ and $\rho_3=\rho_{\delta}$, we get the desired inequality. On the other hand, for the Hausdorff measures inequality, we follow the same reasoning of \cite[Proposition 2.1]{Erraoui Hakiki 2}. Assume that $\mathcal{H}_{\operatorname{euc}}^{\varphi_{d-\alpha}}(F)<\infty$ otherwise there is nothing to prove. Let $\zeta>\mathcal{H}_{\operatorname{euc}}^{\varphi_{d-\alpha}}(F)$ be arbitrary. Then there is a covering of $F$ by open balls $B_{\operatorname{euc}}(x_{n},r_n)$ such that 
\begin{align}
F \subset \bigcup_{n=1}^{\infty} B_{\operatorname{euc}}(x_{n},r_{n}) \quad \text { and } \quad \sum_{n=1}^{\infty} (2r_{n})^{d-\alpha}\,\ell^d(2r_{n}) \leq \zeta.\label{covering for F}
\end{align}
Let $\operatorname{N}_{\delta}(E,r)$ be the smallest number of balls in the metric $\delta$ of radius $r$ by which we can cover $E$. For all $n\geq 1$, let $B_{\delta}(t_{n,j},r_n)$, $j=1,...,\operatorname{N}_{\delta}(E,r_n)$ be the family of balls covering  $E$. It follows that the family $B_{{\delta}}(t_{n,j},r_n)\times B_{\operatorname{euc}}(x_{n},r_n),\,j=1,...,\operatorname{N}_{\delta}(E,r_n)$, $n \geq  1$  covers $E\times F$. Let $P_{\delta}(E,r)$ be the greatest number of disjoint balls $B_{\delta}(x_j,r)$ of radius $r>0$ and centers $x_j\in F$. The left inequality of \eqref{Ahlf-Dav regular cond} ensures that 
\begin{equation}\label{estim packing number}
    \mathsf{c}_{0}^{-1}\, P_{\delta}(E,r)\, r^{\alpha}\leq \sum_{j=1}^{P_{\delta}(E,r)}\mu\left(B_{\delta}(t_j,r)\right)=\mu (G_1)\leq 1 \quad \text{for all $r\in (0,1]$}.
\end{equation}
Using the well known fact that
\begin{equation}\label{well knwn fact}
    \operatorname{N}_{\delta}(E,2\,r)\leq P_{\delta}(E,r),
\end{equation}
we obtain that $\operatorname{N}_{\delta}(E,r)\leq 2^{\alpha}\mathsf{c}_0\,r^{-\alpha}$ for all $r\in (0,1]$.
Hence combining this with \eqref{covering for F} we obtain that 		\begin{align}	\mathcal{H}_{\rho_{\delta}}^{\varphi_d}(E\times F)\leq \sum_{n=1}^{\infty} \sum_{j=1}^{\operatorname{N}_{\delta}(E, r_n)}\left(2 r_{n}\right)^{d}\,\ell^d(2 r_{n}) \leq 2^{2\alpha}\,\mathsf{c}_{0}\, \sum_{n=1}^{\infty} (2r_{n})^{d-\alpha}\,\ell^{d}(2 r_{n}) \leq 2^{2\alpha}\,\mathsf{c}_{0}\, \zeta.\label{covering for E*F}
		\end{align}
Letting $\zeta\downarrow\mathcal{H}_{\operatorname{euc}}^{\varphi_{d-\alpha}}(F)$, the desired inequality follows immediately. 

The result in (ii) is a consequence of \cite[Proposition 2.5]{Erraoui Hakiki 2}, we only need to mention that the capacity term $\mathcal{C}_{\delta_{\log}}^{1}(\cdot)$ considered in \eqref{hitting proba when alpha=d} is equivalent to the capacity term $\mathcal{C}^0_{euc}(\cdot)$ considered in \cite{Erraoui Hakiki 2}. Hence the proof is complete.  
\end{proof}

The next proposition states our undecidability claim with precise assumptions. In particular, any $\alpha$-Ahlfors-David-regular compact set $E$ in $X$'s metric leads to the construction of sets in the target space where one cannot decide whether they are reachable from $E$ based solely on their dimensions.  

\begin{proposition}\label{prop critical case}
Let $X$, $a$, $b$ and $M$ be as in Proposition \ref{bnds in trm of E resp. F}. Let $E\subset [0,1]$ be a $\alpha$-Ahlfors-David regular compact set in the metric $\delta$ with $\alpha \in (0,d)$. Then there exist two compact sets $F_1,F_2\subset [-M,M]^d$ such that $\dim_{\rho_{\delta}}(E\times F_1)=\dim_{\rho_{\delta}}(E\times F_2)=d$\,\, and that 
\begin{align}\label{contre exp critical case}
\mathbb{P}\left\{X(E)\cap F_1 \neq \varnothing \right\}=0 \quad \text{ and } \quad \mathbb{P}\left\{X(E)\cap F_2\neq \varnothing \right\}>0.
\end{align}
\end{proposition}

\begin{remark}
 The previous proposition shows that we can construct image sets leading to undecidability for any compact $\alpha$-Ahlfors-regular  set in the domain of $X$ (relative to  $\delta$), when $\alpha \in (0,d)$. But we are also able to construct examples of undecidable image sets with $\alpha=d$.  Indeed, assume $X$ is a fractional Brownian motion (fBm) with Hurst parameter $H$, and assume $Hd=1$. We show here that in the particular case where $E:=I$ is an interval, the critical case $\dim_{\delta}(E)=\frac{1}{H}=d$ is also undecidable. First note that it was proved in \cite{Dal-Muel-Xia 2017} for a fractional Gaussian random field $X$ restricted on $I_1\times \ldots \times I_k$, for some intervals $I_1,\ldots,I_k$, with Hurst parameter $(H_1,...,H_k)$, that $X$ does not visit points in the critical dimension case $d=Q$ where $Q=H_1^{-1}+...+H_k^{-1}$.
Since our domains are one-dimensional, we apply this to the case of fBm itself, i.e. $k=1$. Let then $X=B^H$ be a $d$-dimensional fBm with $Hd=1$. Let $F_1=\{x\}$ for some fixed point $x\in \mathbb{R}^d$; then evidently $\dim_{\operatorname{euc}}(F_1)=0$ and the aforementioned result \cite{Dal-Muel-Xia 2017} implies $\mathbb{P}\left\{ X(E)\cap F_1\neq \varnothing \right\}=0$. 
On the other hand,  one can easily construct a Borel set $F_2\subset [-1,1]^d$ such that its Euclidean dimension $\dim_{\operatorname{euc}}(F_2)=0$ though it has positive logarithmic capacity $\mathcal{C}_{\delta_{\log}}^{1}(F_2)>0$. Then, by using \eqref{hitting proba when alpha=d}, we obtain $\mathbb{P}\left\{ X(E)\cap F_2\neq \varnothing \right\}>0$. Moreover, the intervals are known to be $1/H$-Ahlfors-David regular in the metric $\delta$; therefore Lemma \ref{lem dim product} ensures that $\dim_{\rho_{\delta}}(E\times F_1)=\dim_{\rho_{\delta}}(E\times F_2)=1/H=d$. This prove the aforementioned undecidability. 
\end{remark}
\begin{proof}[Proof of Proposition \ref{prop critical case}]
First, it turns out that since $E$ is $\alpha$-Ahlfors-David regular in the metric $\delta$, we have the following convenient expression for the $\rho_{\delta}$-dimension of $E \times F$:   
$$
\quad \dim_{\rho_{\delta}}(E\times F) = \dim_{\delta}(E)+\dim_{\operatorname{euc}}(F).
$$
This formula is established in Lemma \ref{lem dim product}, which is stated and proved in the next subsection, though this analysis lemma's proof is self-contained and its result can thus be used here. Therefore, by Proposition \ref{bnds in trm of E resp. F}, recalling the notation $\varphi_{d-\alpha}$ introduced in Item (i) therein, to obtain \eqref{contre exp critical case}, it is sufficient to find $F_1,F_2\subset [-M,M]^d$ such that $\dim_{\operatorname{euc}}(F_1)=\dim_{\operatorname{euc}}(F_2)=d-\alpha$ and that 
\begin{align}\label{nul Haus pos Cap}
\mathcal{H}_{\operatorname{euc}}^{\varphi_{d-\alpha}}(F_1)=0 \quad \text{ and }\quad \mathcal{C}_{\operatorname{euc}}^{d-\alpha}(F_2)>0.
\end{align}
To prove this, we claim that it is sufficient to show the following, which is established in the independent Lemma \ref{lemme48} immediately following the proof of this proposition. Let $\theta>1$ be fixed. There exist two probability measures $\mu_1$ and $\mu_2$ supported by two different compact subsets $F_1$ and $F_2$ of $[-M,M]^d$, such that for some positive constants $\mathsf{c}_5$ and $\mathsf{c}_6$ we have
\begin{align}\label{pseudo Ahlfors condition 1}
\mathsf{c}_5^{-1}\, {\varphi_{d-\alpha}(r)}\, \log^{\theta}(e/r)\leq \mu_1\left(B_{\operatorname{euc}}(x,r)\right)\leq \mathsf{c}_5\, { \varphi_{d-\alpha}(r)}\,\log^{\theta}(e/r) \quad \text{for all $r\in (0,1)$, $x\in F_1$},
\end{align}
and 
\begin{align}\label{pseudo Ahlfors condition 2}
\mathsf{c}_6^{-1}\, r^{d-\alpha}\, \log^{-\theta}(e/r)\leq \mu_2\left(B_{\operatorname{euc}}(x,r)\right)\leq \mathsf{c}_6\, r^{d-\alpha}\, \log^{-\theta}(e/r) \quad \text{for all $r\in (0,1)$, $x\in F_2$}.
\end{align}
We begin by proving our claim \eqref{nul Haus pos Cap} for the compacts $F_1$ and $F_2$ mentioned above.
For all $r\in (0,1)$ and $F\subseteq [-M,M]^d$ 
let $\operatorname{N}_{\operatorname{euc}}(F,r)$ be the minimal number of balls $B_{\operatorname{euc}}(x_j,r)$ of radius $r$ required to cover $F$. By using the lower estimate in \eqref{pseudo Ahlfors condition 1} and the same argument used in \eqref{estim packing number} and \eqref{well knwn fact}, in the Euclidean metric this time,
we deduce that 
\begin{align}\label{estim cover nbr}
    \operatorname{N}_{\operatorname{euc}}(F_1,r)\leq \mathsf{c}_7\,  \left(\varphi_{d-\alpha}(r)\right)^{-1}\,\log^{-\theta}(e/r) \quad \text{ for all $r\in (0,1)$}.
\end{align}
Furthermore, using the definition of the $\varphi_{d-\alpha}$-Hausdorff measure as well as \eqref{estim cover nbr} we infer that
\begin{align}\label{zero hausdorff measure}    \mathcal{H}_{euc}^{\varphi_{d-\alpha}}(F_1)\leq \mathsf{c}_8\,\limsup_{r\rightarrow 0} {\varphi_{d-\alpha}(r)}\, \operatorname{N}_{\operatorname{euc}}(F_1,r)=\limsup_{r\rightarrow 0} \,\log^{-\theta}(e/r)=0,
\end{align}
where $\mathsf{c}_8$ is a positive constant. This gives the first outcome of \eqref{nul Haus pos Cap}. Now, we show that $\mathcal{C}_{d-\alpha}(F_2)>0$, where by definition it is sufficient to prove that $\mathcal{E}_{euc,d-\alpha}(\mu_2)<\infty$, with $\mu_2$ being the measue identified in \eqref{pseudo Ahlfors condition 2}. First notice that the upper bound in \eqref{pseudo Ahlfors condition 2} ensures that $\mu_2$ has no atom. Then for all $x\in F_2$ we have 
\begin{align}\label{sum capacity}
    \dint_{F_2}\frac{\mu_2(dy)}{\lVert x-y\rVert^{d-\alpha}}=\sum_{j=0}^{\infty}\dint_{\{y\,:\,\lVert x-y\rVert \in (\kappa 2^{-(j+1)},\kappa 2^{-j}] \}}\frac{\mu_2(dy)}{\lVert x-y\rVert ^{d-\alpha}} &\leq \sum_{j=0}^{\infty} \kappa^{-(d-\alpha)} 2^{(d-\alpha)\,(j+1)}\mu_2\left(B_{\operatorname{euc}}(t,\kappa\,2^{-j})\right)\nonumber\\
    &\leq 2^{d-\alpha}\,\mathsf{c}_9\,\sum_{j=0}^{\infty}\log^{-\theta}(2^j\,e/\kappa), 
\end{align}
where $\kappa:= \operatorname{diam}_{\operatorname{euc}}(F_2)$ and $\mathsf{c}_9$ depends only on $\theta$, $\alpha$, $\kappa$ and $d$. The last sum is finite since $\theta>1$, and does not depend on $x$ so by integrating with respect to the probability measure $\mu_2(dx)$ we get that $\mathcal{E}_{euc,d-\alpha}(\mu_2)<\infty$, which proves the second outcome of \eqref{nul Haus pos Cap}.

In remains to show that $\dim_{\operatorname{euc}}(F_1)=\dim_{\operatorname{euc}}(F_2)=d-\alpha$. First, notice that the same reasoning as in \eqref{estim cover nbr} and \eqref{zero hausdorff measure} will ensures that  
\begin{align}\label{implication of pseudo-ahlfors condition}    
    \mathcal{H}_{\operatorname{euc}}^{\varphi_{1}}(F_1)<\infty \quad \text{ and }\quad \mathcal{H}_{\operatorname{euc}}^{\varphi_{2}}(F_2)<\infty,
\end{align}
where $\varphi_{1}(r):={\varphi_{d-\alpha}(r)}\log^{\theta}(1/r)$ and $\varphi_{2}(r):=r^{d-\alpha}\log^{-\theta}(1/r)$. Since $\ell^d(\cdot)\,\log^{\theta}(1/\cdot)$ and $\log^{-\theta}(1/\cdot)$ are slowly varying functions and $\lim_{r\rightarrow 0}\ell(r)\in (0,+\infty]$, then 
$$
r^{d-\alpha}=o(\varphi_{1}(r))\quad \text{ and }\quad r^{d-\alpha+\varepsilon}=o(\varphi_{2}(r)) \quad \text{  as $r\rightarrow 0$,}
$$
for all $\varepsilon>0$. This fact combined together with \eqref{implication of pseudo-ahlfors condition} imply that $\mathcal{H}_{\operatorname{euc}}^{d-\alpha}(F_1)=\mathcal{H}_{\operatorname{euc}}^{d-\alpha+\varepsilon}(F_2)=0$ for all $\varepsilon>0$. On the other hand, \eqref{sum capacity} ensures that $\mathcal{E}_{\operatorname{euc},d-\alpha}(\mu_2)<\infty$ and then $\mathcal{C}_{euc}^{d-\alpha}(F_2)>0$. Moreover, repeating the same argument as \eqref{sum capacity}, we obtain that $\mathcal{E}_{\operatorname{euc},d-\alpha-\varepsilon}(\mu_1)<\infty$ and then $\mathcal{C}_{euc}^{d-\alpha-\varepsilon}(F_1)>0$ for all $\varepsilon>0$ small enough. 
Hence, combining all the previous facts we infer than 
$$
d-\alpha-\varepsilon \leq \dim_{\operatorname{euc}}(F_1)\leq d-\alpha \leq \dim_{\operatorname{euc}}(F_2)\leq d-\alpha+\varepsilon \quad \text{for all} \quad \varepsilon>0. 
$$
Since $\varepsilon>0$ is arbitrary, we deduce that $\dim_{\operatorname{euc}}(F_1)=\dim_{\operatorname{euc}}(F_2)=d-\alpha$, which finishes the proof. 

The next lemma, whose proof establishes the existence of measures $\mu_1$ and $\mu_2$ satisfying conditions \eqref{pseudo Ahlfors condition 1} and \eqref{pseudo Ahlfors condition 2}, is enough to conclude the  proof of the proposition.   
\end{proof}

\begin{lemma}\label{lemme48}Let $\alpha\in (0,d)$ and $\theta>1$, then there exist two compact subsets $F_1$ and $F_2$ of $[-M,M]^d$ which respectively, support two probability measures $\mu_1$ and $\mu_2$ satisfying \eqref{pseudo Ahlfors condition 1} and \eqref{pseudo Ahlfors condition 2}.
\end{lemma}
\noindent Before proving this lemma, we give the following key result for constructing the measures in \eqref{pseudo Ahlfors condition 1} and \eqref{pseudo Ahlfors condition 2}, which appears as Proposition 7.4 in \cite{Erraoui Hakiki 2}. The proof of that proposition comes from the procedure for constructing the classical Cantor set and its associated singular continuous distribution function, which is then adapted to a scale that might involve a regularly/slowly varying function in general rather than a power function.
\begin{proposition}\label{EH2}(Appendix B Proposition 7.4 in \cite{Erraoui Hakiki 2})
Let $\psi$ be a function satisfying
\begin{equation}\label{concav condition}
\psi(0)=0\quad \text{and} \quad \psi(2x)< 2\psi(x)\quad \text{for all $x\in (0,x_0)$},
\end{equation}
for some $x_0\in (0,1)$. Then there exists a Borel set $G \subset [0,1]$ which support a probability measure $\nu$ such that 
\begin{equation}
	\mathbf{c}_0^{-1}\,\psi(r)\leq \nu([a-r,a+r])\leq \mathbf{c}_0\,\psi(r) \quad \text{for all $r\in [0,x_0]$ and\, $a\in G$}.
\end{equation}
\end{proposition}
\begin{proof}[Proof of Lemma \ref{lemme48}]
First, let us define the functions $\psi_1(r):=r^{1-\alpha/d}\ell^{1/d}(r)\log^{\theta/d}(e/r)$ and $\psi_2(r):=r^{1-\alpha/d}\log^{-\theta/d}(e/r)$. Since $\alpha<d$, one readily checks that $\psi_i$ for $i=1,2$\, are continuous increasing functions on $(0,1)$ such that \eqref{concav condition} is satisfied.
Therefore, using Proposition \ref{EH2}, there exist two Borel probability measures $\mu_{0,1}$ and $\mu_{0,2}$ supported by two compact subsets $F_{0,1}$ and $F_{0,2}$ of $[0,1]$, respectively, and two positive constants  $\mathbf{c}_1,\mathbf{c}_2$ such that for $i=1,2$ we have
\begin{align}\label{pseudo Ahlf-Dav one dim i}
\mathbf{c}_i^{-1}\,\psi_i(r)\leq \mu_{0,i}([x-r,x+r])\leq \mathbf{c}_i\,\psi_i(r) \quad \text{for all $r\in (0,r_0)$ and\, $x\in F_{0,i}$},
\end{align}
for some $r_0\in (0,1)$.
Now, let $\mu_{i}:=\underset{j=1}{\overset{d}{\otimes}}\mu_{0,i}$ and $F_i:=\underset{j=1}{\overset{d}{\times}}F_{0,i}$  for $i=1,2$. Then using \eqref{pseudo Ahlf-Dav one dim i} and the definition of the measure $\mu_i$, we obtain that 
\begin{equation}\label{pseudo Ahlf-Dav one dim i-d}
    \mathbf{c}_i^{-d}\,\psi_i^d(r)\leq \mu_{i}\left(\overset{d}{\underset{j=1}{\prod}}[x_j-r,x_j+r]\right)\leq \mathbf{c}^d_i\,\psi_i^d(r) \quad \text{for all $r\in (0,1)$ and\, $\left(x_1,\ldots,x_d\right)\in F_{i}$ },
\end{equation}
for $i=1,2$. The fact that the Euclidean norm $\lVert . \rVert_2$ and the maximum norm $\lVert . \rVert_{\infty}$ are equivalent ensures that  \,\eqref{pseudo Ahlfors condition 1} and \eqref{pseudo Ahlfors condition 2} follow with $\mathsf{c}_5$  depending on the constants $\mathbf{c}_1$, $\theta$, $\alpha$, $d$ and $\ell$, and the constant $\mathsf{c}_6$ depending the constants $\mathbf{c}_2$, $\theta$, $\alpha$ and $d$. Hence the proof is complete. 

\end{proof}

\subsection{Co-dimension of the image set $X(E)$}

In this final subsection we consider the so-called stochastic codimension of our image sets. For a random Borel set $K\subset \R^d$ the upper and lower stochastic codimensions of $K$ are defined as follows:
\begin{align}\label{def lwr codim}
\underline{\operatorname{codim}}(K):=\sup\left\{\beta \leq d\,:\, \text{ for all } F\subset \R^d \text{ s.t. } \dim_{euc}(F)<\beta \text{ we have } \mathbb{P}\{K \cap F\neq \varnothing \}=0\right\},
\end{align}
and 
\begin{align}\label{def uppr codim}    \overline{\operatorname{codim}}(K):=\inf\left\{\beta \leq d\,:\, \text{ for all } F\subset \R^d \text{ s.t. } \dim_{euc}(F)>\beta \text{ we have } \mathbb{P}\{K \cap F\neq \varnothing \}>0\right\}.
\end{align}
The above definitions can be found in \cite{Khosh}. Moreover, \cite[Lemma 4.7.1 p. 435]{Khosh} provides the following summary
\begin{equation}\label{summary codim}
\mathbb{P}({K} \cap F \neq \varnothing) \begin{cases}>0, & \text { whenever } \operatorname{dim}_{euc}(F)>\overline{\operatorname{codim}}({K}) \\ =0, & \text { whenever } \operatorname{dim}_{euc}(F)<\underline{\operatorname{codim}}({K})\end{cases}.
\end{equation}
It is worth noting that the upper and lower stochastic codimension of $K$ are not random, even if $K$ is a random set. Notice that $ \underline{\operatorname{codim}}(K)\leq \overline{\operatorname{codim}}(K)$ for all $K$.  Moreover, in the case when $\underline{\operatorname{codim}}(K)=\overline{\operatorname{codim}}(K)$, we write $\operatorname{codim}(K)$ for the common value and call it the stochastic codimension of $K$. 

Let $(Y,\rho)$ be a metric space. We recall that, the upper Minkowski dimension of a Borel set $G\subset Y$, in the metric $\rho$, is defined as
\begin{equation}
	\overline{\dim}_{\rho,M}(G)=\inf\{\alpha : \exists\,\, \mathsf{c}(\alpha)>0\, \text{ such that } N_{\rho}(G,r)\leqslant \mathsf{c}(\alpha)\,r^{-\alpha} \text{ for all } r>0 \}.\label{upper Minkowski}
	\end{equation}
 where $N_{\rho}(G,r)$ is the smallest number of balls of radius $r$ in the metric $\rho$ needed to cover $G$. The following lemma, which shows how Minkowski dimension can be helpful in estimating Hausdorff dimensions, will be useful for the rest of this section, particularly in establishing our formula for the dimension of the cartesian product of two Borel sets, where at least one is Ahlfors-David-regular.
\begin{lemma}\label{lem dim product} 
Let $E\subset [a,b]$ and $F\subset [-M,M]^d$ be two bounded Borel sets. Then we have
\begin{equation}\label{compar dim produit}
\begin{aligned}
 \dim_{\delta}(E)+\dim_{\operatorname{euc}}(F)&\leq \dim_{\rho_{\delta}}(E\times F)\\
 &\leq \left(\overline{\dim}_{\delta,M}(E)+\dim_{\operatorname{euc}}(F)\right)\wedge \left(\dim_{\delta}(E)+\overline{\dim}_{\operatorname{euc},M}(F)\right).
\end{aligned} 
\end{equation}
Moreover, if $E$  (resp. $F$) is Ahlfors-David regular, in the metric  $\delta$ (resp. the Euclidean metric), then 
\begin{equation}\label{effet de Ahlf-Dav}
 \dim_{\delta}(E)=\overline{\dim}_{\delta,M}(E)\quad (\text{resp. }\, \dim_{\operatorname{euc}}(F)=\overline{\dim}_{\operatorname{euc},M}(F)).
\end{equation}
In that case, i.e. when one of $E$ or $F$ is Ahlfors-David regular in its associated metric, we have
\begin{equation}\label{dim prod=sum dim}
\dim_{\rho_{\delta}}(E\times F)= \dim_{\delta}(E)+\dim_{\operatorname{euc}}(F).
\end{equation}
\end{lemma}
\begin{proof}
We start by proving the upper bound in \eqref{compar dim produit}. Let us assume that $\dim_{\delta}(E)>0$ and $\dim_{\operatorname{euc}}(F)>0$ otherwise when one of these dimensions is equal to zero, the result can be readily deduced from the property that the Hausdorff dimension does not increase under projection. Let $\alpha\in (0,\dim_{\delta}(E))$  and $\beta \in (0,\dim_{\operatorname{euc}}(F))$; then $\mathcal{C}_{\operatorname{euc}}^{\beta}(E)>0$ and by Frostman's theorem there is a probability measures $\nu$ supported on $E$ such that 
$$ 
\nu(B_{\delta}(t,r))\leq \mathsf{c}_5\,r^{\alpha}\, \text{ for all $t\in [a,b]$ and $r\in (0,1)$}.
$$
Now, using \cite[Proposition 2.1-i)]{Erraoui Hakiki 2} we have $\mathcal{C}_{\rho_{\delta}}^{\alpha+\beta}(E\times F)\geq \mathsf{c}_6\, \mathcal{C}_{\operatorname{euc}}^{\beta}(F)>0$. 
Hence $\dim_{\rho_{\delta}}(E\times F)\geq \alpha+\beta$. Letting $\alpha \uparrow \dim_{\delta}(E)$ and $\beta \uparrow \dim_{\operatorname{euc}}(F)$, the lower inequality in \eqref{compar dim produit} follows. 

For the upper bound, let $\alpha>\overline{\dim}_{\delta,M}(E)$ and $\beta>\dim_{\operatorname{euc}}(F)$, then $\mathcal{H}_{\operatorname{euc}}^{\beta}(F)=0$ and 
\begin{equation}
    \operatorname{N}_{\delta}(E,r)\leq \mathsf{c}_7\,r^{-\alpha} \quad \text{ for all $r>0$}.
\end{equation}
By using \cite[Proposition 2.1-ii)]{Erraoui Hakiki 2} we obtain
$\mathcal{H}_{\rho_{\delta}}^{\alpha+\beta}(E\times F)\leq \mathsf{c}_8\, \mathcal{H}_{\operatorname{euc}}^{\beta}(F)=0$. Hence $\dim_{\rho_{\delta}}(E\times F)\leq \alpha+\beta$. Letting $\alpha \downarrow \overline{\dim}_{\delta,M}(E)$ and $\beta \downarrow \dim_{\operatorname{euc}}(F)$, the first term of the upper inequality in \eqref{compar dim produit} follows. The second term follows in the same way. For the statement \eqref{effet de Ahlf-Dav} it suffices to go through the same lines of the proof of the Euclidean case, which is shown in \cite[Theorem 5.7 p. $80$]{Mattila}.
The last statement of the lemma follows immediately from its first two statements. 
\end{proof}

We are ready to state and easily prove a formula for the stochatic codimension of our processes' image sets.

\begin{corollary}\label{codimcor} Let $X$ be a $d$-dimensional Gaussian process verifying the commensurability condition $\mathbf{(\Gamma)}$ such that its standard deviation function $\gamma$ satisfies the concavity Hypothesis \eqref{Hyp2} and the mild regularity condition \textbf{$\mathbf{(C_{0+})}$}. Let $E \subset [0,1]$ be a Borel set such that $\dim_{\delta}(E)=\overline{\dim}_{\delta,M}(E)$. Then we have 
\begin{align}\label{cor on codim img}
\operatorname{codim}\left(X(E)\right)=\left(d-\dim_{\delta}(E)\right)\vee 0.    
\end{align}    
\end{corollary}
\begin{proof}
First, using Corollary \ref{polarity in terms of dim of ExF} and Lemma \ref{lem dim product} we obtain that
\begin{align}\label{codim under C_0+ and regular E}
    \mathbb{P}\left\{X(E)\cap F\neq \varnothing \right\}\left\{\begin{array}{ll}
>0 & \text { if } \dim_{\operatorname{euc}}(F)>d-\dim_{\delta}(E) \\
=0 & \text { if } \dim_{\operatorname{euc}}(F)<d-\dim_{\delta}(E)
\end{array}\right..
\end{align}
If $0<\dim_{\delta}(E)<d\,$ then \eqref{codim under C_0+ and regular E} ensures immediately that $\operatorname{codim}(X(E))=d-\dim_{\delta}(E)$. On the other hand, if $\dim_{\delta}(E)\geq d$ then \eqref{codim under C_0+ and regular E} implies that $\mathbb{P}\left\{X(E)\cap F\neq \varnothing\right\}>0$  for all $F\subset [-M,M]^d$ with $\dim_{\operatorname{euc}}(F)>0$, which means that $\overline{\operatorname{codim}}(X(E))=0$. Remains the case when $\dim_{\delta}(E)=\overline{\dim}_{\delta,M}(E)=0$, for which \eqref{codim under C_0+ and regular E} provides that $\mathbb{P}\left\{X(E)\cap F\neq \varnothing\right\}=0$  for all $F\subset [-M,M]^d$ with $\dim_{\operatorname{euc}}(F)<d$. This implies that $\underline{\operatorname{codim}}(X(E))=d$. Hence the proof of \eqref{cor on codim img} is then complete.      
\end{proof}

We finish our paper with a discussion and a conjecture of what may happen when the mild regularity condition \textbf{$\mathbf{(C_{0+})}$} fails to hold. The method of Theorem \ref{Hitting proba} leads to a lack of information on hitting probabilities estimates when that condition fails. For instance, in the logBm scale, i.e. when $\delta(t,s)\asymp \log^{-\beta}(1/|t-s|)$ for some $\beta>1/2$, the method Subsection \ref{1st subsec of 4 sec} leads to a lower bound of $\mathbb{P}\left\{X(E)\cap F\neq \varnothing \right\}$ in terms of the $\rho_{\delta}$-capacity of $E\times F$ with order $d$, and to an upper bound in terms of the $\rho_{\delta}$-Hausdorff measure of $E\times F$ with order $d(1-1/2\beta)$. Namely we have 
\begin{align}
c_1^{-1}\mathcal{C}_{\rho_{\delta}}^d(E\times F)\,  \leq \mathbb{P}\left\{X(E)\cap F\neq \varnothing \right\}\leq c_1\, \mathcal{H}^{d(1-1/2\beta)}_{\rho_{\delta}}(E\times F),
\end{align}
which implies that   
\begin{align}\label{critical log case}
    \mathbb{P}\left\{X(E)\cap F\neq \varnothing \right\}\left\{\begin{array}{ll}
>0 & \text { if } \dim_{\rho_{\delta}}(E\times F)>d \\
=0 & \text { if } \dim_{\rho_{\delta}}(E\times F)<d(1-1/2\beta)
\end{array}\right..
\end{align}
If $E$ is Ahlfors-David regular, by Lemma \ref{lem dim product} we have $\dim_{\rho_{\delta}}(E\times F)=\dim_{\delta}(E)+\dim_{\operatorname{euc}}(F)$. Therefore \eqref{critical log case} takes the following form
\begin{align}\label{critical log case 2}
    \mathbb{P}\left\{X(E)\cap F\neq \varnothing \right\}\left\{\begin{array}{ll}
>0 & \text { if } \dim_{\operatorname{euc}}(F)>d-\dim_{\delta}(E) \\
=0 & \text { if } \dim_{\operatorname{euc}}(F)<d-\dim_{\delta}(E)-d/2\beta
\end{array}\right..
\end{align}
When combining \eqref{def lwr codim}, \eqref{def uppr codim} and \eqref{critical log case 2} we get that 
$$
\overline{\operatorname{codim}}\left(X(E)\right)\leq d-\dim_{\delta}(E) \quad \text{ and } \quad \underline{\operatorname{codim}}\left(X(E)\right)\geq d-\dim_{\delta}(E)-d/2\beta.
$$
On the other hand, it follows from \eqref{dim lower bnd} and \eqref{dim upr bnd logBM}, when $\dim_{\delta}(E)\leq d$, that  $\dim_{\delta}(E)$ and $\dim_{\delta}(E)+d/2\beta$ are lower and upper bounds for $\dim_{\operatorname{euc}} X(E)$, respectively.
Moreover Theorem \ref{0-1 law for image}, {which holds without any regularity assumptions on the standard deviation function $\gamma$},
ensures that $\dim_{\operatorname{euc}}X(E)=\mathbf{c}(E)$  a.s., where $\mathbf{c}(E)$ is a non-random constant depending only on $E$ and on the law of $X$.
Thus in the case of logBm, the constant $\mathbf{c}(E)$ lives in the interval $[\dim_{\delta}(E)\,,\,\dim_{\delta}(E)+d/2\beta]$, which becomes an increasingly precise estimate as one approaches the regularity realm of Condition \textbf{$\mathbf{(C_{0+})}$}. However, we conjecture that the constant $\mathbf{c}(E)$, whose value is unknown for highly irregular processes beyond that realm, is nonetheless directly connected to the image's stochastic codimension. In other words we conjecture the following. 
\begin{conjecture}
Let $X$ be as in Theorem \ref{0-1 law for image}. Let $E\subset [0,1]$ be a Borel set such that $\dim_{\delta}(E)=\overline{\dim}_{\delta,M}(E)\leq \mathbf{c}(E)\leq d$, where $\mathbf{c}(E)$ was defined in that theorem as the almost sure value of $\dim_{\operatorname{euc}}X(E)$. Then 
\begin{align}\label{critical log case 1}
    \mathbb{P}\left\{X(E)\cap F\neq \varnothing \right\}\left\{\begin{array}{ll}
>0 & \text { if } \dim_{\operatorname{euc}}(F)> d- \mathbf{c}(E)\,,\\
=0 & \text { if } \dim_{\operatorname{euc}}(F)< d- \mathbf{c}(E)\,.
\end{array}\right.
\end{align}
In other words, we have the following formula for the stochastic codimension of $X(E)$:
$$
\operatorname{codim}\left( X(E)\right)=d-\mathbf{c}(E).
$$
\end{conjecture}

\bibliographystyle{abbrv}

\end{document}